\documentclass[12pt]{amsart}

\usepackage{amsmath, 
amsfonts, amssymb, amscd}

\newtheorem*{theoremA}{Theorem A}
\newtheorem*{theoremB}{Theorem B}
\swapnumbers
\newtheorem{Theorem}[subsection]{Theorem}
\newtheorem{Lemma}[subsection]{Lemma}
\newtheorem{Proposition}[subsection]{Proposition}
\newtheorem{PropS8}[subsubsection]{Proposition}
\newtheorem{Corollary}[subsection]{Corollary}

\theoremstyle{definition}
\newtheorem{Definition}[subsection]{Definition}

\theoremstyle{remark}
\newtheorem{Remark}[subsection]{Remark}

\numberwithin{equation}{subsection}

\begin{document}

\title[Subrepresentation Theorem]
{Subrepresentation Theorem for \\
$p$-adic Symmetric Spaces} 

\author {Shin-ichi Kato}
\address{Department of Mathematics, Graduate School of Science, 
Kyoto University, Kyoto 606--8502 Japan}
\email{skato@math.kyoto-u.ac.jp}

\author {Keiji Takano}
\address{Department of Arts and Science, 
Akashi National College of Technology, 
679-3 Nishioka, Uozumi-cho, Akashi-City 674-8501 Japan}
\email{takano@akashi.ac.jp}

\date{18/06/2007}


\begin{abstract}  
The notion of relative cuspidality for 
distinguished representations attached to 
$p$-adic symmetric spaces is introduced. 
A characterization of relative cuspidality 
in terms of Jacquet modules is given and 
a generalization of Jacquet's subrepresentation 
theorem to the relative case 
(symmetric space case) is established. 
\end{abstract} 

\maketitle

\section*{Introduction}  

\indent

Let $G$ be a reductive $p$-adic group, 
$\sigma$ an involution on $G$ and 
$H$ the subgroup of all $\sigma$-fixed points in $G$. 
An admissible representation $(\pi,V)$ of $G$ 
is said to be {\it $H$-distinguished} if 
the space $(V^*)^H$ of all $H$-invariant linear 
forms on $V$ is non-zero. Such a representation 
arises as a local component of automorphic 
representations which are of particular 
interest (see \cite{J} for example). From 
representation theoretic point of view, 
distinguished 
representations are the basic object of 
harmonic analysis on the  
$p$-adic symmetric space 
$G/H$. 
By the Frobenius reciprocity, 
these representations can be realized in the space 
of smooth functions on $G/H$. The classification or 
parametrization of such representations would be 
a fundamental problem. Over the real field, harmonic 
analysis on semisimple or reductive 
symmetric spaces has been fully 
developed since 1980's (see \cite{O} and \cite{D}). 
By contrast, the theory over $p$-adic fields 
has not been developed yet. 

In this paper, we suggest a new basic tool 
for the study of 
harmonic analysis on $p$-adic symmetric spaces, 
and establish the relative version 
(for general $p$-adic symmetric 
spaces) of 
Jacquet's subrepresentation theorem 
(for general $p$-adic groups). 

Jacquet's subrepresentation theorem asserts that 
any irreducible admissible representation of a 
reductive $p$-adic group can be embedded in 
a parabolically induced representation. 
The inducing representation of a Levi subgroup 
can be taken as an irreducible {\it cuspidal} 
one, that is, a representation whose 
matrix coefficients are 
compactly supported modulo the center. 
To establish the
relative version  of this theory, first we introduce 
the notion of relative cuspidality as follows. 
An $H$-distinguished representation of $G$ 
is said to be 
{\it $H$-relatively cuspidal} if all the generalized 
matrix coefficients defined by $H$-invariant 
linear forms (which we shall call 
{\it $H$-matrix coefficients}) are 
compactly supported modulo the product of 
$H$ and 
the center of $G$. We use only such 
ones as the
inducing representations. 
Besides, as the inducing subgroups, we use only a 
particular class of parabolic subgroups. 
A parabolic subgroup $P$ is said to be 
{\it $\sigma$-split} if $P$ and $\sigma(P)$ are 
opposite. In this case, 
$M=P\cap\sigma(P)$ is a 
$\sigma$-stable Levi subgroup of $P$ and 
the quotient $M/(M\cap H)$ gives a symmetric space 
of lower rank. Now we state the main theorem 
of this paper. 
\begin{theoremA}[Theorem \ref{7.1}] 
Let $\pi$ be an irreducible $H$-distinguished 
representation of $G$. 
Then there exists a $\sigma$-split 
parabolic subgroup $P$ of $G$ and 
an irreducible $M\cap H$-relatively 
cuspidal representation $\rho$ of 
$M=P\cap\sigma(P)$ such that $\pi$ can 
be embedded in the induced 
representation ${\mathrm{Ind}}_P^G(\rho)$.  
\end{theoremA}\label{theoremA}
This gives a generalization 
of Jacquet's subrepresentation 
theorem in the following sense. 
Take an arbitrary reductive $p$-adic group 
$G_1$ and form the direct product 
$G=G_1\times G_1$. Let $\sigma$ be the involution 
on $G$ which permutes factors. Then the 
$\sigma$-fixed point subgroup in $G$ is the diagonally 
embedded subgroup $\Delta G_1$. 
An irreducible $\Delta G_1$-distinguished representation 
of $G_1\times G_1$ is of the form 
$\pi_1\otimes\widetilde{\pi_1}$ where 
$\pi_1$ is an irreducible admissible 
representation of $G_1$ and 
$\widetilde{\pi_1}$ is the contragredient 
of $\pi_1$. 
It is $\Delta G_1$-relatively 
cuspidal if and only if $\pi_1$ is 
cuspidal in the usual sense. A $\sigma$-split 
parabolic subgroup of $G_1\times G_1$ is of the form 
$P_1\times P_1^-$ where $P_1$ and $P_1^-$ are 
opposite parabolics of $G_1$. Now, Theorem A 
applied to this situation will give an embedding of 
any irreducible admissible representation 
$\pi_1\otimes\widetilde{\pi_1}$ into 
$$
{\mathrm{Ind}}_{P_1\times P_1^-}^{G_1\times G_1}
(\rho_1\otimes\widetilde{\rho_1})
\simeq {\mathrm{Ind}}_{P_1}^{G_1}(\rho_1)
\otimes {\mathrm{Ind}}_{P_1^-}^{G_1}(\widetilde{\rho_1})
$$ 
where $\rho_1$ is an irreducible cuspidal 
representation of a Levi subgroup of $P_1$. In this way, 
Jacquet's subrepresentation theorem can be recovered 
from our theorem. We have an embedding 
$\pi_1\hookrightarrow{\mathrm{Ind}}_{P_1}^{G_1}(\rho_1)$ 
on the first factor, and also 
$\widetilde{\pi_1}\hookrightarrow 
{\mathrm{Ind}}_{P_1^-}^{G_1}(\widetilde{\rho_1})$ on 
the second factor at the same time. 

In the process of obtaining Theorem A, we refer to 
Casselman's proof of Jacquet's 
subrepresentation theorem as a prototype. 
In \cite{C}, a canonical pairing of Jacquet modules 
is constructed. It 
provides a relation of 
the asymptotic behaviors between 
matrix coefficients for a given representation and 
those for its Jacquet modules. 
By this relation, 
the well-known characterization of cuspidality is shown: 
An admissible representation is cuspidal if and only if 
the Jacquet modules along all proper parabolics vanish. 
After this characterization, Jacquet's subrepresentation theorem 
readily follows from the inductive argument 
using the Frobenius reciprocity. 

We consider a relative version of Casselman's canonical 
pairing, which is expected to be a basic tool for 
harmonic analysis of $G/H$. Let $P$ be a 
$\sigma$-split parabolic subgroup with the Levi 
subgroup $M=P\cap\sigma(P)$. 
Let $(\pi_P,V_P)$ denote 
the Jacquet module of $(\pi,V)$ along $P$. 
We shall construct a mapping 
$$
r_P: \left(V^*\right)^H\to 
\left(V_P^*\right)^{M\cap H}
$$ 
of invariant linear forms. This will provide a 
relation of the asymptotic behaviors between 
$H$-matrix coefficients for 
$\pi$ and 
$M\cap H$-matrix coefficients for $\pi_P$. 
Using this relation, we can deduce another 
main theorem which gives a 
characterization of relative cuspidality: 
 
\begin{theoremB}[Theorem \ref{6.9}] 
An $H$-distinguished 
representation $(\pi,V)$ 
of $G$ is
$H$-relatively cuspidal if and only if 
$r_P((V^*)^H)=0$ for any proper 
$\sigma$-split parabolic subgroup $P$ of $G$. 
\end{theoremB} 
Theorem A is a natural 
consequence of this characterization. 

\indent 

This paper is organized as follows. 
Section \ref{S1} is devoted to definitions, notation and 
some elementary properties of 
relatively cuspidal representations. 
In section \ref{S2}, basic notation 
and properties concerning with tori, roots and 
parabolic subgroups associated to 
the involution are prepared. Section \ref{S3} deals with 
the analogue of Cartan decomposition for $p$-adic 
symmetric spaces given independently by 
Benoist-Oh \cite{BO} and 
Delorme-S\'echerre 
\cite{DS}. 
This will be used in an 
important step of the analysis of 
$H$-matrix coefficients. 
In section \ref{S4}, certain families of open compact subgroups 
are introduced. Lemma \ref{4.6} on a property of the subgroups 
in such a family 
will be a key to the 
construction of the mapping $r_P$ in section \ref{S5}. 
After that, we give a result Proposition \ref{5.5} on 
an asymptotic behavior of 
$H$-matrix coefficients. This will be used 
repeatedly in the proof 
of various properties of $r_P$ such as the 
transitivity (Proposition \ref{5.9}). 
Section \ref{S6} is devoted to the proof of Theorem B 
(Theorem \ref{6.9}). 
The proof of Theorem A (Theorem \ref{7.1}) is given 
shortly in section \ref{S7}. Finally, we give several examples of 
relatively cuspidal representations 
in section \ref{S8}.  

\indent 

After this paper was completed, we learned that 
Nathalie Lagier has obtained independently the same 
result as our 
Proposition \ref{5.5} and Proposition \ref{5.6} (1), 
which has appeared in the announcement \cite{L}. 
We are grateful to Patrick Delorme for pointing this out. 

\indent 
 \section {Relatively cuspidal representations} \label{S1}

In this section, we shall introduce 
the notion of 
relative cuspidality for 
distinguished representations and give 
some elementary properties of them. 

\subsection{} \label{1.1}
Let $F$ be a non-archimedean local field with 
the normalized absolute value 
$|\cdot |_F$ 
and the valuation ring $\mathcal O_F$. 
We assume that the 
residual characteristic of $F$ is not
equal to $2$ throughout this paper. 
Let $\mathbf{G}$ be a connected 
reductive group 
over $F$ and 
$\sigma$ an $F$-involution on $\mathbf{G}$. 
The $F$-subgroup $\{h\in\mathbf{G}\,|\,\sigma(h)=h\}$ of 
all $\sigma$-fixed points of 
$\mathbf{G}$ is denoted by $\mathbf{H}$. Let 
$\mathbf{Z}$ be the $F$-split component of 
$\mathbf{G}$, that is, the largest 
$F$-split torus lying in the center of 
$\mathbf{G}$. Note that $\mathbf{Z}$ is 
$\sigma$-stable. The
group $\mathbf{G}(F)$ of $F$-points of
$\mathbf{G}$ is denoted by $G$.
Similarly, for any $F$-subgroup 
$\mathbf{R}$ of
$\mathbf{G}$, we shall write $R=\mathbf{R}(F)$. 

\subsection{} \label{1.2}
Let $(\pi,V)$ be an admissible 
representation of $G$ over $\mathbb C$. It is said to be 
{\it $H$-distinguished} if the space $(V^*)^H$ of all 
$H$-invariant linear forms on $V$ is non-zero. 
For $\lambda\in(V^*)^H$ and $v\in V$, let 
$\varphi_{\lambda,v}$ denote the 
function on $G$ defined by 
$$
\varphi_{\lambda,v}(x)=\langle
\lambda,\pi(x^{-1})v\rangle\quad(x\in G). 
$$
We call such functions {\it $(H,\lambda)$-matrix 
coefficients} of 
$\pi$. These are not the usual matrix coefficients 
of $\pi$, since 
$H$-invariant linear forms are 
not smooth in general. 
Let $C^{\infty}(G/H)$ denote the space of all 
right $H$-invariant locally constant 
$\mathbb C$-valued functions on 
$G$, on which $G$ acts by the left translation. 
Any subrepresentation of 
$C^{\infty}(G/H)$ 
is $H$-distinguished by the linear form 
$\varphi\mapsto\varphi(e)$. 
All the $(H,\lambda)$-matrix coefficients belong to 
$C^{\infty}(G/H)$ with the 
obvious $G$-equivariance 
$$
\varphi_{\lambda,\pi(g)v}(x)=
\varphi_{\lambda,v}(g^{-1}x)
$$
for $x$, $g\in G$. 
Set 
$T_{\lambda}(v)=\varphi_{\lambda,v}$ 
for $\lambda\in (V^*)^H$ and $v\in V$. 
Then the mapping 
$T_{\lambda}: V\to C^{\infty}(G/H)$ 
is a $G$-morphism, 
which is non-zero if and only if $\lambda\neq 0$. 

\begin{Definition} \label{1.3} 
An $H$-distinguished representation 
$(\pi,V)$ is said to be {\it $(H,\lambda)$-relatively cuspidal} 
(for $\lambda\in\left(V^*\right)^H$) 
if all the $(H,\lambda)$-matrix coefficients of $\pi$ are 
compactly supported modulo $ZH$. 
A representation $(\pi,V)$ is said to be 
{\it $H$-relatively cuspidal} if it is  
$(H,\lambda)$-relatively cuspidal for 
all $\lambda\in (V^*)^H$. 
\end{Definition} 

\subsection{}\label{1.4}
Let $C_0^{\infty}(G/H)$ denote 
the subspace of $C^{\infty}(G/H)$ consisting of 
all the functions 
$\varphi\in C^{\infty}(G/H)$ which are 
compactly supported modulo $ZH$. 
It is stable under $G$ from the left. 
An $H$-distinguished representation 
$(\pi,V)$ is $(H,\lambda)$-relatively 
cuspidal for $\lambda\in \left(V^*\right)^H$ if and only if 
the image $T_{\lambda}(V)$ is 
contained in $C_0^{\infty}(G/H)$. 

\subsection{{\it The group case}} \label{1.5} 
Here we see that the notion of relative cuspidality 
introduced above includes the usual cuspidality 
as a special case. 
Take a connected reductive $F$-group 
$\mathbf{G}_1$ and let $\mathbf{G}$ be 
the direct product 
$$
\mathbf{G}=\mathbf{G}_1\times\mathbf{G}_1
$$
equipped with the 
involution 
$$
\sigma(g_1,g_2)=(g_2,g_1). 
$$
The $\sigma$-fixed point subgroup $\mathbf{H}$ is 
the diagonally embedded subgroup 
$$
\Delta\mathbf{G}_1
=\big\{(g,g)\in\mathbf{G}_1\times\mathbf{G}_1
\bigm| g\in\mathbf{G}_1\big\}.
$$ 
The mapping $(g_1,g_2)\mapsto g_1g_2^{-1}$ 
induces a $\mathbf{G}$-equivariant 
isomorphism of 
the quotient $\mathbf{G}/\mathbf{H}
=\left(\mathbf{G}_1\times\mathbf{G}_1\right)
/\Delta\mathbf{G_1}$ onto the 
underlying variety of $\mathbf{G}_1$ on which 
$\mathbf{G}=\mathbf{G}_1\times\mathbf{G}_1$ 
acts by 
multiplication from both sides. 
The $F$-split component 
$\mathbf{Z}$ of $\mathbf{G}$ is of the 
form $\mathbf{Z}_1\times\mathbf{Z}_1$ 
where $\mathbf{Z}_1$ is the $F$-split 
component of $\mathbf{G}_1$. 
Henceforth, this situation is referred to as {\it the group case}. 

Any irreducible admissible
representation 
of $G=G_1\times G_1$ is of the form $\pi_1\otimes
\pi_1'$ 
where $\pi_1$ and $\pi_1'$ are irreducible 
admissible representations of $G_1$. It is 
$\Delta G_1$-distinguished if and only if 
$\pi_1'$ is isomorphic to the contragredient 
$\widetilde{\pi_1}$ of $\pi_1$. The 
natural pairing 
$\langle \, ,\,\rangle_{\widetilde{\pi_1}
\times\pi_1}$ 
of $\widetilde{\pi_1}$ and $\pi_1$ 
defines a non-zero 
$\Delta G_1$-invariant linear form, 
say $\lambda\in\left((\pi_1\otimes\widetilde{\pi_1})^*
\right)^{\Delta G_1}$, which is 
unique up to constant. It is given by 
$$
\langle\lambda,v_1\otimes \widetilde{v}_1\rangle =
\langle \widetilde{v}_1,v_1\rangle_{\widetilde{\pi_1}
\times\pi_1} 
$$
for $v_1\in V_1$ and $\widetilde{v}_1\in\widetilde{V_1}$. 
Observe the relation 
$$
\langle\lambda,(\pi_1\otimes\widetilde{\pi_1})
(g_1,g_2)^{-1}(v_1\otimes\widetilde{v}_1)\rangle =
\langle\lambda,\pi_1(g_1^{-1})v_1\otimes\widetilde{\pi_1}
(g_2^{-1})\widetilde{v}_1\rangle 
$$
$$
=
\langle \widetilde{\pi_1}(g_2^{-1})\widetilde{v}_1,
\pi_1(g_1^{-1})v_1\rangle_{\widetilde{\pi_1}\times
\pi_1}=
\langle \widetilde{v}_1,
\pi_1\left((g_1g_2^{-1})^{-1}\right)v_1
\rangle_{\widetilde{\pi_1}\times\pi_1}
$$
for $g_1$, $g_2\in G_1$ and $v_1\in V_1$, 
$\widetilde{v}_1\in\widetilde{V_1}$. 
As is seen from this,
the $(\Delta G_1,\lambda)$-matrix 
coefficients are identified with the 
usual matrix coefficients of 
$\pi_1$ through the mapping 
$(g_1,g_2)\mapsto g_1g_2^{-1}$. 
  
Recall that an admissible representation 
is said to be cuspidal if all the usual matrix coefficients 
are compactly supported modulo the center. 
Now it is obvious from the above identification 
that 
$\pi_1\otimes\widetilde{\pi_1}$ is 
$(\Delta G_1,\lambda)$-relatively (actually $\Delta G_1$-relatively) 
cuspidal 
if and only if 
$\pi_1$ is cuspidal as a 
representation of $G_1$. 

\subsection{} \label{1.6} 
For an admissible representation 
$(\pi,V)$ of $G$ and a 
quasi-character $\omega$ of $Z$, put 
$$
V_{\omega}=\big\{v\in V\bigm|\pi(z)v=\omega(z)v
\text{ for all }z\in Z\big\}. 
$$
We call a subrepresentation of $V$ an 
$\omega$-subrepresentation if 
it is contained in $V_{\omega}$, 
and call $V$ an $\omega$-representation if 
$V=V_{\omega}$. 

\begin{Lemma} \label{1.7} 
Let $(\pi,V)$ be an admissible 
representation of $G$ of finite length. 
Then there 
exists a non-trivial quotient 
representation of $V$ which is isomorphic 
to an $\omega$-subrepresentation of $V$ 
for some quasi-character 
$\omega$ of $Z$. 
\end{Lemma} 
\begin{proof}
Recall from \cite[2.1.9]{C} the direct sum 
decomposition $V=\bigoplus_{\omega} V_{\omega,\infty}$ 
where $\omega$ runs over a 
finite set of quasi-characters of $Z$,  
$$
V_{\omega,n}=\big\{v\in V\bigm|
(\pi(z)-\omega(z))^nv=0\,\text{ for all } z\in Z\big\}
$$
and 
$$
V_{\omega,\infty}=\bigcup_{n}V_{\omega,n}. 
$$
It is enough to consider the case 
$V=V_{\omega,\infty}$ for some quasi-character 
$\omega$ of $Z$. 
If $V=V_{\omega,1}=V_{\omega}$, 
there is nothing to prove. 
If not, there exist elements 
$z_0\in Z$ and $v_0\in V$ 
such that 
$$
\pi(z_0)v_0-\omega(z_0)v_0\neq 0.
$$
Set $\phi_1=\pi(z_0)-\omega(z_0){\mathrm{id}}_V: V\to V$. 
This is a non-zero $G$-morphism with non-trivial kernel. 
Let $V_1$ be the image of $\phi_1$. 
It is a proper $G$-submodule of $V$. 
If $V_1\subset V_{\omega}$, then 
$V/{\rm{Ker}}(\phi_1)\simeq V_1$ will give the 
desired quotient. 
If $V_1\not\subset V_{\omega}$, there are $z_1\in Z$ and 
$v_1\in V_1$ such that 
$$
\pi(z_1)v_1-\omega(z_1)v_1\neq 0.
$$
Set $\phi_2=\left(\pi(z_1)-\omega(z_1){\mathrm{id}}_V
\right)|_{V_1}
: V_1\to
V_1$.  Again this is a non-zero 
$G$-morphism with non-trivial kernel. Let $V_2$ be the image of
$\phi_2$, which is a proper $G$-submodule of $V_1$. 
In this way, we can construct a non-zero proper $G$-submodule 
$V_k$ of $V_{k-1}$ with a surjective 
$G$-morphism $V_{k-1}\stackrel{\phi_k}{\twoheadrightarrow} V_k$ 
inductively as long as $V_{k-1}$ is not contained in $V_{\omega}$. 
We obtain a decreasing sequence 
$$
V\supsetneqq V_1\supsetneqq V_2\supsetneqq\cdots 
\supsetneqq V_k\neq 0,
$$
of $G$-submodules, together with surjective 
$G$-morphisms 
$$
V\stackrel{\phi_1}{\twoheadrightarrow} V_1
\stackrel{\phi_2}{\twoheadrightarrow}  V_2
\twoheadrightarrow \cdots 
\stackrel{\phi_k}{\twoheadrightarrow}  V_k. 
$$ 
We must have $V_k\subset V_{\omega}$ for some $k$ 
by the finiteness of length of $V$. 
Now $V/{\rm{Ker}}(\phi_k\circ\cdots\circ\phi_1)\simeq 
V_k$ will give 
the desired quotient. 
\end{proof}

\subsection{} \label{1.8} 
For a quasi-character $\omega$ of $Z$ which is 
trivial on $Z\cap H$, let 
$C_{0,\omega}^{\infty}(G/H)$ denote 
the space of all $\varphi
\in C_0^{\infty}(G/H)$ satisfying 
$$
\varphi(zg)=\omega^{-1}(z)\varphi(g)
$$
for $g\in G$ and $z\in Z$. 
If an $\omega$-representation is 
$H$-distinguished, then 
we must have $\omega|_{Z\cap H}\equiv 1$. 
An $H$-distinguished 
$\omega$-representation 
$(\pi, V)$ is $(H,\lambda)$-relatively cuspidal 
for $\lambda\in\left(V^*\right)^H$ 
if and only if the image 
$T_{\lambda}(V)$ is contained in 
$C_{0,\omega}^{\infty}(G/H)$. 

If $\omega$ is unitary, then any subrepresentation 
of $C_{0,\omega}^{\infty}(G/H)$ 
is pre-unitary by the $G$-invariant 
hermitian inner product 
$$
\langle\langle\,\varphi,\,\psi\,\rangle\rangle
=\int_{G/ZH} 
\varphi(\dot{g})
\overline{\psi(\dot{g})}d\dot{g}\quad
(\varphi,\,\psi\in C_{0,\omega}^{\infty}(G/H))
$$
where $d\dot{g}$ denotes a 
$G$-invariant measure on the quotient 
$G/ZH$. 

\subsection{} \label{1.9} 
Let ${\mathcal X}(G, \mathbb R^{\times}_{>0})$ be the set of 
all positive real valued quasi-characters of $G$ and put 
$$
{\mathcal X}(G/H, \mathbb R^{\times}_{>0})
=\{\,\chi\in {\mathcal X}(G, \mathbb R^{\times}_{>0})
\bigm|\chi|_H\equiv 1\,\}. 
$$
We define ${\mathcal X}(Z, \mathbb R^{\times}_{>0})$ 
and ${\mathcal X}(Z/Z\cap H, \mathbb R^{\times}_{>0})$ 
similarly. 

\begin{Lemma} \label{1.10} 
$(1)$ The restriction map 
$$
{\mathcal X}(G, \mathbb R^{\times}_{>0})
\stackrel{\rm{res}}{\longrightarrow }
{\mathcal X}(Z, \mathbb R^{\times}_{>0}),\quad 
{\rm{res}}(\chi)=\chi|_{Z}\,\,
\left(
\chi\in {\mathcal X}(G, \mathbb R^{\times}_{>0})
\right)
$$
is bijective. 

$(2)$ The above restriction map sends 
${\mathcal X}(G/H, \mathbb R^{\times}_{>0})$ 
bijectively onto ${\mathcal X}(Z/Z\cap H, \mathbb R^{\times}_{>0})$.
\end{Lemma} 
\begin{proof} 
(1) is well-known (e.g., \cite[5.2.5]{C}). For (2), 
first it is 
obvious that 
$\chi|_{Z}\in 
{\mathcal X}(Z/Z\cap H, \mathbb R^{\times}_{>0})$ for all 
$\chi\in 
{\mathcal X}(G/H, \mathbb R^{\times}_{>0})$. On the other hand, 
for a given $\omega\in 
{\mathcal X}(Z/Z\cap H, \mathbb R^{\times}_{>0})$, 
let us take the element 
$\chi\in  {\mathcal X}(G, \mathbb R^{\times}_{>0})$ such that 
$\chi|_{Z}=\omega$ by (1). Since $z\sigma(z)\in Z\cap H$ 
for all $z\in Z$, we have 
$\omega(z\sigma(z))=1$ so that 
$\omega\circ\sigma=\omega^{-1}$. 
Thus we have $(\chi\circ\sigma)|_Z=(\chi|_Z)\circ\sigma=
(\chi|_Z)^{-1}=\chi^{-1}|_Z$, which shows that 
$\chi\circ\sigma=\chi^{-1}$ on $G$ by (1). 
This proves that $\chi|_H\equiv 1$ 
since $\chi(h)=\chi^{-1}(h)\in\Bbb R^{\times}_{>0}$ 
for all $h\in H$.
\end{proof}

\begin{Proposition} \label{1.11} 
Any finitely generated $(H,\lambda)$-relatively 
cuspidal representation of $G$
has a non-trivial $H$-distinguished irreducible quotient. 
\end{Proposition}
\begin{proof} 
Let $(\pi,V)$ be a finitely generated 
admissible representation with a 
non-zero $\lambda\in(V^*)^H$ such that 
$T_{\lambda}(V)\subset 
C_0^{\infty}(G/H)$. The quotient 
$V/{\rm{Ker}}(T_{\lambda})$ is finitely 
generated, hence is of finite length by 
\cite[6.3.10]{C}. Apply Lemma \ref{1.7} to obtain 
a quotient $(\overline{\pi}, \overline{V})$ of 
$V/{\rm{Ker}}(T_{\lambda})$,
hence of $V$, which is isomorphic 
to an $\omega$-subrepresentation of 
$V/{\rm{Ker}}(T_{\lambda})$ 
for some $\omega$. Then $(\overline{\pi}, \overline{V})$ 
is regarded as a subrepresentation of 
$C_{0,\omega}^{\infty}(G/H)$ through 
$V/{\rm{Ker}}(T_{\lambda})\simeq 
T_{\lambda}(V)\subset C_0^{\infty}(G/H)$. 
Let $\Re(\omega)$ be the quasi-character of $Z$ 
defined by 
$$
\Re(\omega)(z)=\left(\omega(z)\overline{\omega(z)}
\right)^{1/2} 
$$
for $z\in Z$. This belongs to 
${\mathcal X}(Z/Z\cap H, \mathbb R^{\times}_{>0})$ and 
$\Im(\omega):=\omega\Re(\omega)^{-1}$ is unitary. 
Now, take the element 
$\chi\in {\mathcal X}(G/H, \mathbb R^{\times}_{>0})$
such that 
$\chi|_Z=\Re(\omega)^{-1}$ by 
Lemma \ref{1.10} (2) and consider the representation 
$\chi\cdot\overline{\pi}$ on $\overline{V}$ defined by
$$ 
\left(\chi\cdot\overline{\pi}\right)(g)=
\chi(g)\overline{\pi}(g)
$$
for $g\in G$. It is 
regarded as a subrepresentation of 
$C_{0,\Im(\omega)}
^{\infty}(G/H)$, which is pre-unitary 
by \ref{1.8}. Thus $(\chi\cdot\overline{\pi},\overline{V})$ is decomposed into a
direct sum of finitely many irreducible subrepresentations 
(see \cite[2.1.14]{C}). 
The decomposition for the action $\chi\cdot\overline{\pi}$ also yields 
that for the action $\overline{\pi}$, hence the claim follows. 
\end{proof}

\section{$\sigma$-split parabolic subgroups} \label{S2} 

 In this section, we shall prepare basic notation 
and properties concerning with tori, roots and 
parabolic subgroups associated to $p$-adic 
symmetric spaces. 

\subsection{} \label{2.1} 
An $F$-split torus $\mathbf{S}$ in $\mathbf{G}$ 
is said to be 
{\it $(\sigma,F)$-split} if $\sigma(s)=s^{-1}$ for 
all $s\in\mathbf{S}$. Take a 
maximal $(\sigma,F)$-split torus 
$\mathbf{S}_0$ of $\mathbf{G}$ and a 
maximal $F$-split torus 
$\mathbf{A}_{\emptyset}$ 
of $\mathbf{G}$ containing 
$\mathbf{S}_0$. Then $\mathbf{A}_{\emptyset}$ 
is $\sigma$-stable by \cite[4.5]{HW}. 
Put 
$\mathbf{M}_0=Z_{\mathbf{G}}(\mathbf{S}_0)$ and 
$\mathbf{M}_{\emptyset}=Z_{\mathbf{G}}
(\mathbf{A}_{\emptyset})$. 
Clearly these are also $\sigma$-stable. 

Let $X^*_F(\mathbf{A}_{\emptyset})$ be 
the free $\Bbb Z$-module of all 
$F$-rational 
characters of 
$\mathbf{A}_{\emptyset}$ on which $\sigma$ acts
naturally. Let $\Phi=
\Phi(\mathbf{G},\mathbf{A}_{\emptyset})$ be the root
system of $(\mathbf{G},\mathbf{A}_{\emptyset})$. 
Then $\sigma$
leaves $\Phi$ stable. Set 
$$
\Phi_{\sigma}=\{\,\alpha\in\Phi\,\left|\right.
\,\sigma(\alpha)=\alpha\,\}.
$$
This is a closed subsystem of 
$\Phi$. As in \cite[1.6]{HH}, we 
choose a {\it 
$\sigma$-basis} $\Delta$ 
(and the corresponding set 
$\Phi^+$ of positive roots) of $\Phi$ so that 
$\sigma(\alpha)\notin\Phi^+$ for 
all $\alpha\in\Phi^+\setminus\Phi_{\sigma}$. 
Put
$\Delta_{\sigma}=\Delta\cap
\Phi_{\sigma}$. This gives a 
basis of the subsystem $\Phi_{\sigma}$. 
Let $w_{\Delta_{\sigma}}$ be the longest element of the 
Weyl group of $\Phi_{\sigma}$ with respect to 
the basis $\Delta_{\sigma}$. As in 
\cite[1.7]{HH}, let $\sigma^*$
be the involution on 
$X^*_F(\mathbf{A}_{\emptyset})$ given by 
$\sigma^*=-\sigma\circ w_{\Delta_{\sigma}}$. 
It is known that $\sigma^*$ leaves $\Delta$ stable.

\subsection{} \label{2.2} 
Let $\mathbf{S}_0\subset
\mathbf{A}_{\emptyset}$ and $\Delta$ be as above. 
Let $\mathbf{P}_{\emptyset}$ ($\supset\mathbf{A}
_{\emptyset}$) 
be the minimal parabolic
$F$-subgroup of
$\mathbf{G}$ corresponding to $\Delta$. 
In this paper, parabolic $F$-subgroups
containing such 
$\mathbf{P}_{\emptyset}$ are said to be 
{\it standard with respect to} 
$(\mathbf{S}_0, 
\mathbf{A}_{\emptyset}, \Delta)$. 
They correspond to subsets of 
$\Delta$. Let us fix the notation. 
For a subset $I\subset\Delta$, let 
$\mathbf{P}_I$ be the standard parabolic
$F$-subgroup of
$\mathbf{G}$
corresponding to $I$ 
with the unipotent
radical $\mathbf{U}_I$. Let 
$\mathbf{A}_I$ denote the identity
component of the intersection
of all the kernels of $\alpha: 
\mathbf{A}_{\emptyset}\to\Bbb G_m$, $\alpha\in I$, 
and set 
$\mathbf{M}_I=Z_{\mathbf{G}}(\mathbf{A}_I)$. 
We have a Levi decomposition $\mathbf{P}_I=
\mathbf{M}_I\ltimes \mathbf{U}_I$. 
The torus $\mathbf{A}_I$ 
is the $F$-split component of $\mathbf{M}_I$. 
Let $\mathbf{P}_I^-$ 
be the unique parabolic $F$-subgroup 
of $\mathbf{G}$ such that 
$\mathbf{P}_I\cap \mathbf{P}_I^-=
\mathbf{M}_I$ and $\mathbf{U}_I^-$ the 
unipotent radical 
of $\mathbf{P}_I^-$. 
Later we shall drop the subscript $I$ if there 
is no fear 
of confusion. 

\subsection{} \label{2.3}
A parabolic $F$-subgroup $\mathbf{P}$ 
of $\mathbf{G}$ 
is said 
to be {\it $\sigma$-split} if 
$\mathbf{P}$ and 
$\sigma(\mathbf{P})$ are 
opposite. In such a case, $\mathbf{P}\cap
\sigma(\mathbf{P})$ gives a $\sigma$-stable 
Levi subgroup of $\mathbf{P}$. 
Let us fix $\mathbf{S}_0$, 
$\mathbf{A}_{\emptyset}$ and $\Delta$ as above. 
If $\mathbf{P}_I$ is $\sigma$-split for 
a subset $I\subset\Delta$, then 
we must have $\sigma(\mathbf{P}_I)=
\mathbf{P}_I^-$, that is, 
$\mathbf{P}_I\cap
\sigma(\mathbf{P}_I)=
\mathbf{M}_I$ (see \cite[4.6]{HW}). 
Hence we also have 
$\sigma(\mathbf{U}_I)=
\mathbf{U}_I^-$. 
Further, recall from 
\cite[2.6]{HH} that 
${\mathbf{P}}_I$ is $\sigma$-split if and only if 
$\Delta_{\sigma}\subset I$ and the subsystem $\Phi_I$ 
of $\Phi$ generated by $I$ is $\sigma$-stable.
There is an alternative description for 
${\mathbf{P}}_I$ to be $\sigma$-split as follows: 
If $I$ contains $\Delta_{\sigma}$, then
$w_{\Delta_{\sigma}}$ leaves $\Phi_I$ stable 
so that  
\begin{eqnarray*} 
\Phi_I\,\,\text{is }\sigma\text{-stable}
&\Longleftrightarrow&\Phi_I\,\,\text{is }
\sigma^*\,(=-\sigma\circ w_{\Delta_{\sigma}})
\text{-stable} \\
&\Longleftrightarrow& I\,\,\text{is }
\sigma^*\text{-stable} \\ 
& &\text{(since }
\sigma^*(\Delta)=\Delta\, \text{ and }
\Delta\cap\Phi_I=I).
\end{eqnarray*} 
We say that a subset of $\Delta$ is 
{\it $\sigma$-split} 
if it contains $\Delta_{\sigma}$ and 
is $\sigma^*$-stable. Thus 
$\mathbf{P}_I$ is $\sigma$-split if 
and only if $I$ is a 
$\sigma$-split subset of $\Delta$. 

Note that every $\sigma$-split parabolic 
$F$-subgroup arises as a standard 
$\sigma$-split one  
for a suitable choice of 
$\mathbf{S}_0$, 
$\mathbf{A}_{\emptyset}$ and $\Delta$. 
See \cite[4.6, 4.7]{HW} and also 
Lemma \ref{2.5} (1) below. 

\subsection{} \label{2.4}
The subset $I=\Delta_{\sigma}$ is a 
minimal $\sigma$-split subset of $\Delta$. 
The corresponding parabolic $F$-subgroup 
$\mathbf{P}_{\Delta_{\sigma}}$ is denoted 
by $\mathbf{P}_0$. 
This is a minimal $\sigma$-split parabolic 
$F$-subgroup of $\mathbf{G}$ and the 
$\sigma$-stable Levi subgroup 
$\mathbf{M}_{\Delta_{\sigma}}=\mathbf{P}_0\cap\sigma(\mathbf{P}_0)$ 
coincides with $\mathbf{M}_0=Z_{\mathbf{G}}(\mathbf{S}_0)$ 
(see \cite[4.7]{HW}). Put 
$\mathbf{P}_0^-=\mathbf{P}_{\Delta_{\sigma}}^-$, 
$\mathbf{U}_0=\mathbf{U}_{\Delta_{\sigma}}$ 
and 
$\mathbf{U}_0^-=\mathbf{U}_{\Delta_{\sigma}}^-$. 

\begin{Lemma} \label{2.5}
Let us fix 
$\mathbf{S}_0$, 
$\mathbf{A}_{\emptyset}$ and $\Delta$ 
as above. 

$(1)$ Any $\sigma$-split parabolic $F$-subgroup 
of $\mathbf{G}$ is of the form 
$\gamma^{-1}\mathbf{P}_I\gamma$ for some 
$\sigma$-split subset $I\subset\Delta$ and 
$\gamma\in \left(\mathbf{M}_0
\mathbf{H}\right)(F)$. 

$(2)$ If $\left(\mathbf{M}_0
\mathbf{H}\right)(F)=
M_0H$, then any $\sigma$-split parabolic
$F$-subgroup of $\mathbf{G}$ is 
$H$-conjugate to a standard 
$\sigma$-split one.
\end{Lemma}  
\begin{proof} (2) follows directly from (1). For 
(1), let $\mathbf{P}\subset\mathbf{G}$ be 
any $\sigma$-split parabolic $F$-subgroup. 
As in \cite[4.9, 4.11]{HW}, we can take 
an element 
$\gamma\in (\mathbf{P}_{\emptyset}\mathbf{H})(F)$ 
such that $\gamma\mathbf{P}\gamma^{-1}$ contains 
$\mathbf{P}_{\emptyset}$ and is $\sigma$-split. 
It is enough to see that 
$p^{-1}\gamma\in
(\mathbf{M}_0\mathbf{H})(F)$ for some 
$p\in P_0$. Since 
$\mathbf{P}_{\emptyset}\mathbf{H}
=\mathbf{P}_0\mathbf{H}$ by \cite[4.8]{HW}, we have 
$$
x:=\gamma\sigma(\gamma)^{-1}\in 
(\mathbf{U}_0\mathbf{M}_0\mathbf{U}_0^-)(F)
=U_0M_0U^-_0. 
$$
Express $x$ as $x=u_0m_0u_0^-$ 
with $u_0\in U_0$, $m_0\in M_0$ and 
$u_0^-\in U^-_0$. Since 
$x=\sigma(x)^{-1}$, we have 
$$
u_0m_0u_0^-=\sigma(u_0^-)^{-1}\sigma(m_0)^{-1}
\sigma(u_0)^{-1}. 
$$ 
By the uniqueness of the expression of 
elements of $U_0M_0U^-_0$, 
we must have 
$$
m_0=\sigma(m_0)^{-1},\quad 
u_0^-=\sigma(u_o)^{-1}. 
$$
Now it is seen that 
$$
u_0^{-1}\gamma\sigma(u_0^{-1}\gamma)^{-1}
=u_0^{-1}x\sigma(u_0)=
u_0^{-1}x(u_0^-)^{-1}=m_0\in M_0. 
$$
This shows that
$u_0^{-1}\gamma\in(\mathbf{M}_0\mathbf{H})(F)$. 
\end{proof} 

\subsection{} \label{2.6}
For each $\alpha\in\Delta\setminus
\Delta_{\sigma}$, 
we put $I_{\alpha}=\Delta\setminus\{\alpha,\sigma^*(\alpha)\}$. 
(It may happen that $\alpha=\sigma^*(\alpha)$.) 
This is a maximal proper
$\sigma$-split subset of $\Delta$. 
Hence all the maximal $\sigma$-split parabolic 
$F$-subgroups of $\mathbf{G}$ are 
of the form $\gamma^{-1}\mathbf{P}_{I_{\alpha}}\gamma$ 
for some $\alpha\in\Delta\setminus\Delta_{\sigma}$ and 
$\gamma\in(\mathbf{M}_0\mathbf{H})(F)$. 

\subsection{} \label{2.7}
For a $\sigma$-split subset 
$I\subset\Delta$, the identity component of 
$\mathbf{A}_I\cap\mathbf{S}_0$ is 
denoted by $\mathbf{S}_I$. 
We shall call $\mathbf{S}_I$ {\it the $(\sigma,F)$-split component of 
$\mathbf{M}_I$} (or {\it of $\mathbf{P}_I$}). 
Note that the $(\sigma,F)$-split 
component $\mathbf{S}_{\Delta_{\sigma}}$ of 
the minimal $\sigma$-split parabolic 
$F$-subgroup $\mathbf{P}_{\Delta_{\sigma}}=
\mathbf{P}_0$ coincides with 
$\mathbf{S}_0$, the maximal 
$(\sigma,F)$-split torus we choose. 

For a positive real number $\varepsilon\leqq 1$, set 
$$
A_I^-(\varepsilon)
=\{\,a\in A_I\,\bigl|\bigr.\,
|a^{\alpha}|_F\leqq\varepsilon\,\,
(\alpha\in\Delta\setminus I)\}
$$
and 
$$ 
S_I^-(\varepsilon)=S_I\cap 
A_I^-(\varepsilon)
=\{\,s\in S_I\,\bigl|\bigr.\,
|s^{\alpha}|_F\leqq\varepsilon\,\,
(\alpha\in\Delta\setminus I)\}. 
$$
Also put 
\begin{eqnarray*}
S_I^+(\varepsilon)&=&\{\,s\in S_I\,\bigl|\bigr.\,
s^{-1}\in S_I^-(\varepsilon)\,\} \\
&=&\{\,s\in S_I\,\bigl|\bigr.\,
|s^{\alpha}|_F\geqq\varepsilon\,\,
(\alpha\in\Delta\setminus I)\}.
\end{eqnarray*}
Since $S_I^-(\varepsilon)\subset 
A_I^-(\varepsilon)$, the following lemma is 
apparent (see \cite[1.4.3]{C}). 

\begin{Lemma} \label{2.8}
Let 
$I\subset\Delta$ be a $\sigma$-split subset. 
For any two open compact subgroups 
$U_1$, $U_2$ of $U_I$, 
there exists a positive real 
number $\varepsilon\leqq 1$ such that 
$sU_1s^{-1}\subset U_2$ 
for all $s\in S_I^-(\varepsilon)$. 
\end{Lemma}

\indent

Next we shall give a lemma on 
the modulus character of 
$\sigma$-split parabolic subgroups. 

\begin{Lemma} \label{2.9} 
Let $\mathbf{P}=
\mathbf{M}\ltimes\mathbf{U}$ be a 
$\sigma$-split parabolic $F$-subgroup 
of $\mathbf{G}$. 
Then the modulus character 
$\delta_P$ of 
$P$ is trivial on 
$M\cap H$. 
\end{Lemma} 
\begin{proof} 
Let $\mathbf{A}$ be 
the $F$-split 
component of $\mathbf{M}$. 
Since $\delta_P$ is positive real valued, 
it is determined by the values on 
$A$ according to Lemma \ref{1.10} (1). 
It is  enough to show that
$\delta_P$ is trivial on 
$A\cap H$ by (2) of Lemma \ref{1.10}. 
We may suppose that $P=P_I$ for a $\sigma$-split subset $I\subset\Delta$. 
Then we have 
$$
\delta_P(a)=\prod_{\alpha\in\Phi^+\setminus\Phi_I}
\left|a^{\alpha}\right|_F^{m_{\alpha}}
\quad(a\in A)
$$
where $m_{\alpha}$ denotes the dimension of 
the root space attached to $\alpha\in\Phi$. 
Note that $m_{\alpha}=m_{\sigma(\alpha)}$ since $\sigma$ 
maps the root space attached to $\alpha$ isomorphically 
onto the one attached to $\sigma(\alpha)$. 
If further $a\in A\cap H$, we have 
\begin{eqnarray*} 
\delta^2(a)&=&
\delta(a)\delta(\sigma(a))
=\prod_{\alpha\in\Phi^+\setminus\Phi_I}
\left|a^{\alpha}\right|_F^{m_{\alpha}}
\prod_{\alpha\in\Phi^+\setminus\Phi_I}
\left|a^{\sigma(\alpha)}\right|_F^{m_{\alpha}}\\
&=&\prod_{\alpha\in\Phi^+\setminus\Phi_I}
\left|a^{\alpha}\right|_F^{m_{\alpha}}
\prod_{\alpha\in\Phi^+\setminus\Phi_I}
\left|a^{\sigma(\alpha)}\right|_F^{m_{\sigma(\alpha)}}
\\
&=&\prod_{\alpha\in\Phi^+\setminus\Phi_I}
\left|a^{\alpha}\right|_F^{m_{\alpha}}
\prod_{\beta\in\Phi^+\setminus\Phi_I}
\left|a^{-\beta}\right|_F^{m_{\beta}}=1
\end{eqnarray*}
and the claim follows. 
\end{proof} 

\section{The analogue of Cartan decomposition} \label{S3}

We shall recall from \cite{BO} and \cite{DS} a certain 
decomposition 
theorem related to 
$p$-adic symmetric spaces and 
give a variant of it for our later use. 
It will play an important role in the study of the support of 
$H$-matrix coefficients. 

\subsection{} \label{3.1} 
Let us fix a triple 
$(\mathbf{S}_0, \mathbf{A}_{\emptyset}, 
\Delta)$ as in \ref{2.1}. 
For simplicity, put 
$S_0^+=S_{\Delta_{\sigma}}^+(1)$ so that 
$$
S_0^+
=\{\,s\in\mathbf{S}_0(F)\,\bigm|\,
|s^{\alpha}|_F\geqq 1\,
(\alpha\in\Delta\setminus\Delta_{\sigma})\}. 
$$
Let $\{\mathbf{S}_0^{(j)}\}_{j\in J}$ be a set of representatives 
for $H$-conjugacy classes of 
maximal $(\sigma,F)$-split tori of $\mathbf{G}$. In this case, 
the index set $J$ is finite by \cite[2.12]{HH}. 
For each $j\in J$, there exists an element $y_j\in 
\left(\mathbf{M}_{\emptyset}\mathbf{H}\right)(F)$ 
such that $
y_j^{-1}\mathbf{S}_0y_j=
\mathbf{S}_0^{(j)}$ by \cite[10.3]{HW}. 
Set 
$$
W_G(\mathbf{S}_0^{(j)})=
N_G(\mathbf{S}_0^{(j)})/Z_G(\mathbf{S}_0^{(j)})
$$
and 
$$
W_H(\mathbf{S}_0^{(j)})=
N_H(\mathbf{S}_0^{(j)})/Z_H(\mathbf{S}_0^{(j)}). 
$$ 
Naturally $W_H(\mathbf{S}_0^{(j)})$ is regarded as 
a subgroup of 
the finite (Weyl) group $W_G(\mathbf{S}_0^{(j)})$. 
Let ${\mathcal N}_j\subset N_G(\mathbf{S}_0^{(j)})$ be 
a complete set of
representatives  of 
$W_G(\mathbf{S}_0^{(j)})/W_H(\mathbf{S}_0^{(j)})$ 
in $N_G(\mathbf{S}_0^{(j)})$. 

\begin{Theorem}[{\cite[Theorem 3.10]{DS}}] \label{3.2}
There exists a compact subset 
$\Omega$ of $G$ such that 
$$ 
G=\bigcup_{j\in J}\bigcup_{n\in{\mathcal N}_j} 
\Omega S_0^+
y_j n H. 
$$
\end{Theorem} 

\noindent 
An essentially equivalent assertion is given also in 
\cite[Theorem 1.1]{BO}. 

\subsection{} \label{3.3}
Now for each $j\in J$, we have 
$$
N_G(\mathbf{S}_0^{(j)})=N_G(y_j^{-1}\mathbf{S}_0y_j)
=y_j^{-1}N_G(\mathbf{S}_0)y_j. 
$$
Put 
$$
{\mathcal N}'_j=\{y_jny_j^{-1}\,|\, n\in{\mathcal N}_j\}. 
$$
We have ${\mathcal N}'_j\subset N_G(\mathbf{S}_0)$ for all 
$j\in J$. The above decomposition is written as 
$$
G=\bigcup_{j\in J}\bigcup_{n\in{\mathcal N}'_j} 
\Omega S_0^+
ny_j H.
$$ 
Furthermore, we have 
$n\sigma(n)^{-1}\in\mathbf{M}_0(F)$ 
for any $n\in N_G(\mathbf{S}_0)$ since 
$n\sigma(n)^{-1}$ centralizes $\mathbf{S}_0$. 
Consequently, for any 
$n\in N_G(\mathbf{S}_0)$ and 
$y\in\left(\mathbf{M}_0\mathbf{H}\right)(F)$, 
we have 
\begin{eqnarray*} 
ny\sigma(ny)^{-1}&=&n(y\sigma(y)^{-1})\sigma(n)^{-1}\\ 
&=&n(y\sigma(y)^{-1})n^{-1}\cdot n\sigma(n)^{-1}\in M_0. 
\end{eqnarray*} 
This means that 
$ny\in\left(\mathbf{M}_0\mathbf{H}\right)(F)$. 
As a result, we have a variant of 
Theorem \ref{3.2} in the following form.  

\begin{Corollary} \label{3.4}
There exists a compact subset $\Omega$ of $G$ and 
a finite subset 
$\Gamma$ of $\left(\mathbf{M}_0\mathbf{H}\right)(F)$ 
such that 
$$
G=\Omega S_0^+\Gamma H. 
$$
\end{Corollary} 

\section{Preliminaries on open compact subgroups } \label{S4}

For a later use on the study of 
invariant linear forms on Jacquet modules, 
we shall construct a particular family of 
open compact subgroups 
adapted to the involution. 

\subsection{} \label{4.1}
In this section, all parabolic $F$-subgroups 
are standard ones 
with respect to some fixed data 
$(\mathbf{S}_0, 
\mathbf{A}_{\emptyset}, \Delta)$ as in 
\ref{2.3}, but the subscripts $I$ ($\subset\Delta$) 
will be 
omitted. For a parabolic 
$F$-subgroup $\mathbf{P}=\mathbf{M}\ltimes
\mathbf{U}$ of $\mathbf{G}$ 
and an open compact subgroup
$K$ of $G$, we set 
$$
U_K=U\cap K,\quad M_K=M\cap K,\quad U^-_K=U^-\cap K.
$$ 
Note that 
$\sigma(M_K)=M_K$ and 
$\sigma(U_K)=U^-_K$ if $K$ is 
$\sigma$-stable and $\mathbf{P}$ is $\sigma$-split. 

\subsection{} \label{4.2}
For each choice of a maximal $F$-split torus 
$\mathbf{A}_{\emptyset}$ and a basis $\Delta$ of 
the root system of $(\mathbf{G},
\mathbf{A}_{\emptyset})$, there is a decreasing 
sequence $\{K'_n\}_{n\geqq 0}$ of open compact 
subgroups of $G$ satisfying 
the following properties as in \cite[1.4.4]{C}:  
\begin{enumerate} 
\item\label{4.2(1)}
It gives a 
fundamental system of open neighborhoods of the 
identity $e$ in $G$. 
\item\label{4.2(2)}
For each $n\geqq 1$, the subgroup 
$K_n'$ is normal in $K_0'$ and 
the quotient 
$K_n'/K_{n+1}'$ is a finite abelian $p$-group 
where $p$ denotes the residual characteristic of $F$. 
\item \label{4.2(3)}
For each $K'=K'_n$ $(n\geqq 1)$ and each standard 
parabolic $F$-subgroup $\mathbf{P}=
\mathbf{M}\ltimes\mathbf{U}$ of 
$\mathbf{G}$ (corresponding 
to $(\mathbf{A}_{\emptyset},\Delta)$) 
with 
the $F$-split component $\mathbf{A}$, 
the product map 
$$
U^-_{K'}\times 
M_{K'}\times U_{K'}\to 
K'
$$ 
is bijective and 
$$
a\cdot U_{K'}\cdot a^{-1}\subset U_{K'}, \quad 
a^{-1}\cdot U^-_{K'}\cdot a\subset U^-_{K'}
$$
for all $a\in A^-(1)$. 
\item \label{4.2(4)}
For each standard parabolic 
subgroup $\mathbf{P}=
\mathbf{M}\ltimes\mathbf{U}$ of $G$, the family 
$\{M\cap K_n'\}_{n\geqq 0}$ enjoys the same 
properties as \eqref{4.2(1)}--\eqref{4.2(3)} above for the group $M$.   
\end{enumerate}
Here the latter half of \eqref{4.2(2)} is not apparent in 
\cite[1.4.4]{C}. However, in the $F$-split case, 
the argument in \cite[2.2.11]{G} 
shows that each quotient 
$K'_n/K'_{n+1}$ is isomorphic to 
the additive 
group of a Lie algebra over the 
residue field of $F$. 
The general case is reduced to the 
$F$-split case as in \cite[1.4.4]{C}. 

\begin{Lemma} \label{4.3}
Fix a data 
$(\mathbf{S}_0, 
\mathbf{A}_{\emptyset}, \Delta)$ as in 
$\ref{2.1}$ and let $\{K_n'\}_{n\geqq
0}$ be as above. 
Put 
$K_n=K_n'\cap\sigma(K_n')$ for each $n$. 
Then the family 
$\{K_n\}_{n\geqq 0}$ is a decreasing sequence of 
$\sigma$-stable open compact subgroups of 
$G$ satisfying the following properties: 
\begin{enumerate}
\item \label{4.3(1)}
It gives a 
fundamental system of open neighborhoods of the 
identity $e$ in $G$. 
\item \label{4.3(2)}
For each $n\geqq 1$, the subgroup 
$K_n$ is normal in $K_0$ and 
the quotient 
$K_n/K_{n+1}$ is a finite abelian $p$-group. 
\item \label{4.3(3)}
For each $K=K_n$ $(n\geqq 1)$ and each 
$\sigma$-split parabolic subgroup 
$\mathbf{P}=
\mathbf{M}\ltimes\mathbf{U}$ of $\mathbf{G}$ 
(standard with respect to 
$(\mathbf{S}_0, 
\mathbf{A}_{\emptyset}, \Delta)$) 
with 
the $(\sigma,F)$-split component $\mathbf{S}$, 
the product map 
$$
U^-_{K}\times 
M_{K}\times U_{K}\to 
K
$$ 
is bijective and 
$$
s\cdot U_{K}\cdot s^{-1}\subset U_{K},\quad
s^{-1}\cdot U^-_{K}\cdot s\subset U^-_{K}
$$
for all $s\in S^-(1)$. 
\item \label{4.3(4)}
For each $\sigma$-split parabolic 
subgroup $\mathbf{P}=
\mathbf{M}\ltimes\mathbf{U}$ 
of $\mathbf{G}$ standard with 
respect to 
$(\mathbf{S}_0, 
\mathbf{A}_{\emptyset}, \Delta)$, 
the family 
$\{M\cap K_n\}_{n\geqq 0}$ enjoys the same 
properties as \eqref{4.3(1)}--\eqref{4.3(3)} above for the group $M$.   
\end{enumerate}
\end{Lemma} 
\begin{proof} 
These are derived directly from the corresponding 
properties of $K_n'$ in \ref{4.2}. 
First, note that $k$ belongs to 
$K_n$ if and only if both $k$ and 
$\sigma(k)$ belong to $K_n'$. Now, 
\eqref{4.3(1)} is obvious. For \eqref{4.3(2)}, take any 
$k_1$, $k_2\in K_n$ and consider their 
commutator. By \ref{4.2} \eqref{4.2(2)}, we
have  
$k_1k_2k_1^{-1}k_2^{-1}\in K_{n+1}'$ and 
$$
\sigma(k_1)\sigma(k_2)
\sigma(k_1)^{-1}\sigma(k_2)^{-1}
=\sigma(k_1k_2k_1^{-1}k_2^{-1})\in K_{n+1}',
$$ 
hence $k_1k_2k_1^{-1}k_2^{-1}\in K_{n+1}$. 
For \eqref{4.3(3)}, it is sufficient to show 
the  surjectivity of the product map. 
Given $k\in K=K_n$, 
decompose $k$ and $\sigma(k)^{-1}\in K'=K_n'$ 
as 
$$
k=u_1^-m_1u_1,\quad 
\sigma(k)^{-1}=u_2^-m_2u_2
$$
where $u_i^-\in U_{K'}^-$, 
$m_i\in M_{K'}$, $u_i\in U_{K'}$ ($i=1,2$) 
by \ref{4.2} \eqref{4.2(3)}. 
Then we have 
$$
\sigma(u_2)u_1^-=\sigma(m_2)^{-1}
\sigma(u_2^-)^{-1}u_1^{-1}m_1^{-1}\in U^-\cap P=\{e\},
$$
which shows that 
$u_1^-=\sigma(u_2)^{-1}$ and in turn, 
$m_1=\sigma(m_2)^{-1}$, $u_1=\sigma(u_2^-)^{-1}$. 
Now we have 
$$
u_1^-=\sigma(u_2)^{-1}\in\left(
U^-\cap K'\right)\cap 
\left(\sigma(U)\cap\sigma(K')
\right)
=U^-\cap K
$$
and similarly 
$m_1\in M\cap K$, $u_1\in U\cap K$. 
Finally, \eqref{4.3(4)} follows once 
\eqref{4.3(1)}--\eqref{4.3(3)} are verified.  
\end{proof}

We say that a family $\{K_n\}_{n\geqq 0}$ of 
$\sigma$-stable open compact subgroups of $G$ 
is {\it adapted to} 
$(\mathbf{S}_0, 
\mathbf{A}_{\emptyset}, \Delta)$ if 
it satisfies the above properties \eqref{4.3(1)}--\eqref{4.3(4)}. 
By the bijectivity of the product map 
of \eqref{4.3(3)}, any element $k\in K=K_n$ can be 
written uniquely as 
$$
k=u^-\cdot m\cdot u^+,\quad u^-\in U_K^-,\,
m\in M_K,\,u^+\in U_K.
$$    
Such an expression is called {\it the Iwahori 
factorization with respect to $P$}. 
 
\subsection{} \label{4.4}
Let $p$ be an odd prime and 
$C$ a finite abelian $p$-group. 
The homomorphism 
$a\mapsto a^2$ of $C$ into itself is bijective. 
The inverse map is denoted by $a\mapsto a^{1/2}$. 
Let $\sigma$ be an involution 
on $C$. If $a\in C$ satisfies the condition 
$\sigma(a)^{-1}=a$, then 
$b=a^{1/2}$ is an element of $C$ such that 
$a=b\sigma(b)^{-1}$. 

\begin{Lemma} \label{4.5}
Let $p$ be an odd prime, 
$K$ a totally disconnected 
compact group and 
$\sigma$ an involution on $K$. 
Suppose that $K$ has a decreasing 
sequence 
$$
K=K_1\supset K_2\supset K_3\supset \cdots 
\supset K_n\supset \cdots
$$
of $\sigma$-stable open normal 
subgroups such that 

$(\rm i)$ 
$K_n/K_{n+1}$ is a finite 
abelian $p$-group for each $n\geqq 1$, 

\noindent and 

$(\rm {ii})$ $\bigcap_{n\geqq 1} K_n=\{e\}$. 

\noindent Then, 
for any $k\in K$ satisfying the condition 
$\sigma(k)^{-1}=k$, there exists an element 
$k'\in K$ such that $k=k'\sigma(k')^{-1}$. 
\end{Lemma} 
\begin{proof} 
Our discussion below is similar to that of 
\cite[Theorem 6.8]{PR}. Look at the involution 
on the quotient $K_1/K_2$ induced by $\sigma$. 
If a given $k\in K=K_1$ satisfies $\sigma(k)^{-1}=k$, then 
by \ref{4.4}, there exists an 
element $y_1\in K_1$ such that 
$$
y_1\sigma(y_1)^{-1}\equiv k\mod K_2. 
$$
Set $k_2=y_1^{-1}k\sigma(y_1)$. It is an element 
of $K_2$ and satisfies 
$$
\sigma(k_2)^{-1}=y_1^{-1}\sigma(k)^{-1}\sigma(y_1)
=y_1^{-1}k\sigma(y_1)
=k_2.
$$
Looking at the involution on $K_2/K_3$ 
induced by $\sigma$, there exists an element
$y_2\in K_2$ such that 
$$
y_2\sigma(y_2)^{-1}\equiv k_2\mod K_3. 
$$
Set $k_3=y_2^{-1}k_2\sigma(y_2)\in K_3$. 
We have $\sigma(k_3)^{-1}=k_3$. 
In this way, we can take 
$k_{n+1}$, $y_{n+1}\in K_{n+1}$ from 
$k_n$, $y_n\in K_{n}$ by the rules 
$$
k_{n+1}=y_n^{-1}k_n\sigma(y_n),\quad 
y_{n+1}\sigma(y_{n+1})^{-1}\equiv k_{n+1}
\mod K_{n+2}. 
$$
Consider the sequence $\{z_n\}$ in $K$ 
defined by 
$z_n=y_1y_2\cdots y_n$. 
It has a 
subsequence $\{z_{n_{\nu}}\}$ 
which converges to an element, 
say $k'$, of $K$. Note that 
$$
k_{n+1}=y_n^{-1}\cdots y_1^{-1}k\sigma(y_1)\cdots\sigma(y_n)
=z_n^{-1}k\sigma(z_n)\in K_{n+1}. 
$$
Thus 
$$
(k')^{-1}k\sigma(k')\in \bigcap_{\nu} K_{n_{\nu}+1}=\{e\},
$$
which shows the claim. 
\end{proof}

\begin{Lemma} \label{4.6}
Let $\mathbf{P}=\mathbf{M}
\ltimes\mathbf{U}$ be a $\sigma$-split 
parabolic $F$-subgroup which is standard with 
respect to $(\mathbf{S}_0, 
\mathbf{A}_{\emptyset}, \Delta)$ and 
$K=K_n$ $(n\geqq 1)$ a $\sigma$-stable 
open compact subgroup of $G$ from the family 
$\{K_n\}$ adapted to 
$(\mathbf{S}_0, 
\mathbf{A}_{\emptyset}, \Delta)$. 
Then, 
$$
U_K\subset HM_KU_K^-.
$$
\end{Lemma}
\begin{proof} 
For a given $u\in U_K$, consider the element $k:=
u^{-1}\sigma(u)
\in K$. By using the Iwahori 
factorization with respect to $P$, 
express $k$ as 
$$
k=u^-\cdot m\cdot u^+,\quad u^-\in U_K^-,\,
m\in M_K,\,u^+\in U_K.
$$
Since $\sigma(k)^{-1}=k$, 
we have 
$$
\sigma(u^+)^{-1}\cdot\sigma(m)^{-1}\cdot 
\sigma(u^-)^{-1}=u^-\cdot m\cdot u^+
$$ 
so that 
$u^+=\sigma(u^-)^{-1}$ and 
$\sigma(m)^{-1}=m$. Here note that 
the group $M_K$ satisfies the assumption of Lemma 
\ref{4.5}. 
Thus we can take 
an element $m'\in M_K$ such that 
$m=m'\sigma(m')^{-1}$. 
As a result,  we have 
$$
u^{-1}\sigma(u)=k=u^-\cdot m'\sigma(m')^{-1}
\cdot \sigma(u^-)^{-1}=
u^-m'\sigma(u^-m')^{-1}.
$$
This shows that $uu^-m'\in H$, hence
$u\in  HM_KU_K^-$. 
\end{proof}

\section{Invariant linear forms 
on Jacquet modules} \label{S5}

For a given 
$H$-invariant linear form $\lambda$ on an 
$H$-distinguished representation, 
we shall construct a canonical linear form 
$r_P(\lambda)$ on 
the Jacquet module along each 
$\sigma$-split 
parabolic $F$-subgroup $\mathbf{P}$ by 
using Casselman's canonical lifting. 
It turns out that $r_P(\lambda)$ is 
$M\cap H$-invariant where 
$M=P\cap\sigma(P)$. 
We have a useful relation between 
$(H,\lambda)$-matrix coefficients 
and $(M\cap H, r_P(\lambda))$-matrix 
coefficients on the $(\sigma,F)$-split 
component of $\mathbf{P}$. 

\subsection{} \label{5.1}
From now on, we say briefly that $P$ is a $\sigma$-split 
parabolic subgroup of $G$ if it is the 
group of $F$-points of a $\sigma$-split 
parabolic $F$-subgroup $\mathbf{P}$ of 
$\mathbf{G}$. Also we say that $S$ is the 
$(\sigma,F)$-split 
component of $P$ if it is the group of 
$F$-points of the $(\sigma,F)$-split component 
$\mathbf{S}=\mathbf{S}_I$ of 
$\mathbf{P}=\mathbf{P}_I$, and so on. 
As a Levi subgroup of a $\sigma$-split parabolic 
subgroup, we always take the $\sigma$-stable one as in 
\ref{2.3}. 

Let $(\pi,V)$ be an admissible 
representation of $G$ and $P=M\ltimes\nolinebreak U$ a 
$\sigma$-split parabolic subgroup with the 
$(\sigma,F)$-split component $S$. 
We regard $P$ as a standard one with respect to 
a suitable choice of $(\mathbf{S}_0, 
\mathbf{A}_{\emptyset}, \Delta)$. 
Let $(\pi_P, V_P)$ denote 
the normalized Jacquet module 
of $(\pi,V)$ along $P$: The space $V_P$ is given as 
the quotient $V/V(U)$ where 
$V(U)$ denotes the subspace of $V$ spanned by 
all the elements of the form $\pi(u)v-v$, 
$v\in V$, $u\in U$. The action $\pi_P$ of 
$M$ on $V_P$ is normalized so that
$$
\pi_P(m)j_P(v)=
\delta_P^{-1/2}(m)j_P(\pi(m)v)
$$
for $m\in M$ and $v\in V$. 
Here $j_P:V\to V_P$ denotes 
the  canonical projection and 
$\delta_P$ the modulus of $P$. 

\subsection{} \label{5.2}
Let us recall the construction of 
Casselman's canonical lifting. 
For a compact subgroup $K$ of $G$, let $V^K$ be the subspace 
of $V$ of all $K$-fixed vectors and 
${\mathcal P}_K: V\to V^K$ the projection operator 
given by 
$$
{\mathcal P}_K(v)=\dfrac{1}{{\rm{vol}}(K)}
\int_K\pi(k)vdk\quad(v\in V).
$$ 
For a compact subgroup $U_1$ of $U$, set 
$$
V(U_1)=\left\{\,v\in V\,\Biggm| \int_{U_1}\pi(u)vdu=0\,\right\}.
$$
It is known that 
$V(U)$ is the union of all
$V(U_1)$ where 
$U_1$ ranges over all compact subgroups of $U$. 
Now, for a given $\overline{v}\in V_P$, take an 
open compact subgroup $K=K_n$ from the family 
adapted to 
$(\mathbf{S}_0, 
\mathbf{A}_{\emptyset}, \Delta)$ 
such that 
$\overline{v}\in(V_P)^{M_K}$. Next choose an open compact
subgroup $U_1$ of $U$ with $V^K\cap V(U)\subset 
V(U_1)$. 
Finally we take a positive real number 
$\varepsilon\leqq 1$ such that 
$sU_1s^{-1}\subset U_K$ for all 
$s\in S^-(\varepsilon)$ by Lemma \ref{2.8}. 
Since 
$S^-(\varepsilon)$ is contained in 
$A^-(\varepsilon)$, we may replace 
$A^-(\varepsilon)$ in the argument of {\cite[\S 4]{C}} by 
$S^-(\varepsilon)$. 
We have an isomorphism 
$$
{\mathcal
P}_K\left(\pi(s)V^K\right)
\stackrel{\simeq}{\longrightarrow}\left(V_P\right)^{M_K} 
$$
by the restriction of the canonical projection 
$j_P:V\to V_P$ as in \cite[4.1.4]{C}. 
The element $v\in 
{\mathcal P}_K\left(\pi(s)V^K\right)$ satisfying 
$j_P(v)=\overline{v}$ is called {\it the canonical 
lift of $\overline{v}\in V_P$ with respect to $K$}. 
It depends on the choice of $K$, but not on 
$U_0$ and $\varepsilon$. 
If $v'$ is another canonical lift  of $\overline{v}$ 
with respect to $K'\subset K$, 
then we have 
$$
v'\in V^{M_KU^-_K},\quad
v={\mathcal P}_K(v')={\mathcal P}_{U_K}(v') 
$$
by \cite[4.1.8]{C}. 

\begin{Proposition} \label{5.3}
Let $\lambda$ be an
$H$-invariant linear  form on an admissible 
representation $(\pi,V)$ of $G$ and 
$P$ a $\sigma$-split parabolic subgroup 
standard with respect to
$(\mathbf{S}_0,\mathbf{A}_{\emptyset},\Delta)$. 

$(1)$ For $K=K_n$ $(n\geqq 1)$ in the family 
$\{K_n\}$ adapted to
$(\mathbf{S}_0,\mathbf{A}_{\emptyset},\Delta)$ 
and $v\in V^{M_KU^-_K}$, 
one has 
$$
\langle\,\lambda\,,\,v\,\rangle=
\langle\,\lambda\,,\,{\mathcal P}_{U_K}(v)\,\rangle.
$$ 

$(2)$ For two canonical lifts 
$v$, $v'\in V$ of a given $\overline{v}\in V_P$, 
one has 
$$
\langle\,\lambda\,,\,v\,\rangle=
\langle\,\lambda\,,\,v'\,\rangle. 
$$
\end{Proposition} 
\begin{proof} 
(1) Let $U_1$ be an open compact subgroup 
of $U_K$ which fixes $v$. Then we have 
\begin{eqnarray*}
\langle\,\lambda\,,\,{\mathcal P}_{U_K}(v)\,\rangle
&=&\langle\,\lambda\,,\,\dfrac{1}{{\rm{vol}}(U_K)}
\int_{U_K}\pi(u)vdu\,\rangle \\
&=&\langle\,\lambda\,,\,
\dfrac{{\rm{vol}}(U_1)}{{\rm{vol}}(U_K)}
\sum_{u_i\in U_K/U_1}\pi(u_i)v\,\rangle. 
\end{eqnarray*}
Let us express each $u_i\in U_K$ as $u_i=h_im_iu^-_i$ where 
$h_i\in H$, 
$m_i\in M_K$ and $u^-_i\in U^-_K$ by Lemma \ref{4.6}. Then 
this is equal to 
$$
\langle\,\lambda\,,\,
\dfrac{{\rm{vol}}(U_1)}{{\rm{vol}}(U_K)}
\sum_i\pi(h_im_iu^-_i)v\,\rangle
=\langle\,\lambda\,,\,v\,\rangle
$$
since $\lambda$ is $H$-invariant and 
$v\in V^{M_KU^-_K}$. 

(2) Assume that $v$ (resp. $v'$) is the canonical lift of 
$\overline{v}\in V_P$ with respect to $K$ (resp. $K'$). 
It is enough to consider the case where 
$K'\subset K$. 
By the remark preceding this proposition, we have 
$$
v'\in V^{M_KU^-_K},\quad
v={\mathcal P}_{U_K}(v'). 
$$
It follows from (1) that 
$$
\langle\,\lambda\,,\,v\,\rangle=
\langle\,\lambda\,,\,{\mathcal P}_{U_K}(v')\,\rangle 
=
\langle\,\lambda\,,\,v'\,\rangle. 
$$ 
\end{proof}

\subsection{} \label{5.4}
After the above proposition, we can define 
a linear form $r_P(\lambda)$ on the Jacquet module $V_P$ 
along a $\sigma$-split parabolic subgroup $P$ by 
$$
\langle r_P(\lambda),\overline{v}\rangle
=\langle\lambda,v\rangle
$$
where $v\in V$ is a canonical lift of $\overline{v}
\in V_P$. 

\begin{Proposition} \label{5.5}
Let $\lambda$ be an
$H$-invariant linear  form on an admissible 
representation $(\pi,V)$ of $G$ and 
$P=M\ltimes U$ a 
$\sigma$-split  parabolic subgroup of $G$ with the 
$(\sigma,F)$-split component $S$.   

$(1)$ For each $v\in V$, 
there exists 
a positive real number $\varepsilon\leqq 1$ such 
that  
$$
\langle\lambda,\pi(s)v\rangle
=
\delta_P^{1/2}(s)
\langle r_P(\lambda),\pi_P(s)j_P(v)\rangle
$$
for any $s\in S^-(\varepsilon)$. 

$(2)$ Assume that 
$\overline{\lambda}$ is a linear form on $V_P$ having the
following property: For each $v\in V$, 
there exists a positive
real number $\varepsilon
\leqq 1$ such that   
$$
\langle \lambda,\pi(s)v\rangle
=\delta_P(s)^{1/2}\langle \overline{\lambda},\pi_P(s)j_P(v)
\rangle
$$
for any $s\in S^-(\varepsilon)$. 
Then $\overline{\lambda}$ 
coincides with $r_P(\lambda)$. 
\end{Proposition}
\begin{proof} 
(1) For a given $v\in V$, choose 
an open compact 
subgroup $K=K_n$ from the adapted family 
such that $v\in V^K$. 
Take an open compact subgroup $U_1$ of $U_K$ with 
$V^K\cap V(U)\subset V(U_1)$. Let $\varepsilon\leqq 1$ be a 
positive real number such that 
$sU_1s^{-1}$ is contained in $U_K$ 
for all $s\in S^-(\varepsilon)$. Then, by the 
Iwahori factorization with respect to $P$, we have 
$\pi(s)v\in V^{M_KU^-_K}$ so that 
$$
j_P(\pi(s)v)=
j_P({\mathcal P}_{U_K}(\pi(s)v))=
j_P({\mathcal P}_K(\pi(s)v)) 
$$
for all $s\in S^-(\varepsilon)$. On the other hand, 
since $j_P(v)\in (V_P)^{M_K}$ and $s$ is central in $M$, 
we have $\pi_P(s)j_P(v)\in (V_P)^{M_K}$ so that 
$$
j_P(\pi(s)v)=\delta_P^{1/2}(s)\pi_P(s)j_P(v)\in 
(V_P)^{M_K}. 
$$
These relations show that ${\mathcal P}_K(\pi(s)v)=
{\mathcal P}_{U_K}(\pi(s)v)$ is a canonical lift 
of $\delta_P^{1/2}(s)\pi_P(s)j_P(v)$. Thus, by
definition we have 
$$
\langle r_P(\lambda),\delta_P(s)^{1/2}\pi_P(s)j_P(v)
\rangle=\langle \lambda,{\mathcal P}_{U_K}
\left(\pi(s)v\right)\rangle 
$$
and the right hand side is equal to 
$\langle \lambda,\pi(s)v\rangle$ by Lemma \ref{5.3} (1). 

(2) Take an open compact subgroup $K=K_n$ from 
the adapted family. Let 
$\varepsilon$ be a positive real 
number such that 
$$
\langle \lambda,\pi(s)v\rangle
=\delta_P(s)^{1/2}\langle \overline{\lambda},\pi_P(s)j_P(v)
\rangle
$$
for all $s\in S^-(\varepsilon)$ and all 
$v\in V^K$. This is possible since 
$V^K$ is finite dimensional. 
We may choose this $\varepsilon$
so that the space 
${\mathcal P}_K(\pi(s)V^K)$ is independent of 
$s\in S^-(\varepsilon)$ by \cite[4.1.6]{C}. 
Since 
$\pi(s)v\in V^{M_KU^-_K}$, we have 
$$
\langle \lambda,\pi(s)v\rangle=
\langle \lambda,{\mathcal P}_{U_K}(\pi(s)v)\rangle
=\langle \lambda,{\mathcal P}_K(\pi(s)v)\rangle
$$
for all $s\in S^-(\varepsilon)$. 
On the other hand, 
${\mathcal P}_K(\pi(s)v)$ is a canonical
lift  of $\delta_P^{1/2}(s)\pi_P(s)j_P(v)$ again, 
so we have 
$$
\langle \lambda,{\mathcal P}_K(\pi(s)v)\rangle=
\delta_P^{1/2}(s)
\langle r_P(\lambda),\pi_P(s)j_P(v)\rangle.
$$
As a result, we have 
$$
\langle \overline{\lambda},\pi_P(s)j_P(v)
\rangle=\langle r_P(\lambda),\pi_P(s)j_P(v)\rangle
$$
for any $s\in S^-(\varepsilon)$. Since $j_P(V^K)=
(V_P)^{M_K}$ (see \cite[3.3.3]{C}), 
this shows that $\overline{\lambda}=
r_P(\lambda)$ on $(V_P)^{M_K}$. Letting $K$ vary in 
the adapted family, we 
conclude that 
$\overline{\lambda}=
r_P(\lambda)$ on $V_P$. 
\end{proof}

\begin{Proposition} \label{5.6}
Let $\lambda$ be an
$H$-invariant linear  form on an admissible 
representation $(\pi,V)$ of $G$ and 
$P=M\ltimes U$ a 
$\sigma$-split  parabolic subgroup of $G$. 

$(1)$ The linear form 
$r_P(\lambda)$ on $V_P$ is 
$M\cap H$-invariant. 

$(2)$ The mapping $r_P: (V^*)^H\to (V_P^*)^{M\cap H}$ 
is linear. 
\end{Proposition} 
\begin{proof}
(1) For a given $m\in M\cap H$, set 
$\overline{\lambda}
=r_P(\lambda)\circ\pi_P(m)$. By (1) of Proposition \ref{5.5} 
for the vector 
$\pi(m)v$, we can take 
$0<\varepsilon\leqq 1$ such that 
$$
\langle\lambda,\pi(s)\pi(m)v\rangle
=
\delta_P^{1/2}(s)
\langle r_P(\lambda),\pi_P(s)j_P(\pi(m)v)\rangle
$$
for all 
$s\in S^-(\varepsilon)$. 
Since $s$ is central in $M$, the left hand side 
is equal to 
$$
\langle\lambda,\pi(m)\pi(s)v\rangle=
\langle\lambda,\pi(s)v\rangle,
$$
while the right hand side is
equal to 
\begin{eqnarray*} 
& &\delta_P^{1/2}(s)\delta_P^{1/2}(m)
\langle r_P(\lambda),\pi_P(s)\pi_P(m)j_P(v)\rangle \\
& &=\delta_P^{1/2}(s)
\langle r_P(\lambda),\pi_P(m)\pi_P(s)j_P(v)\rangle
=\delta_P^{1/2}(s)
\langle \overline{\lambda},\pi_P(s)j_P(v)\rangle
\end{eqnarray*}
by Lemma \ref{2.9}. Thus 
$\overline{\lambda}$ has the property of 
Proposition \ref{5.5} (2), which implies that 
$r_P(\lambda)$ coincides with 
$\overline{\lambda}=r_P(\lambda)\circ\pi_P(m)$. 

(2) For any $\lambda_1$, $\lambda_2\in \left(V^*\right)^H$ 
and $c_1$, $c_2\in\Bbb C$, it is easy to see that 
$c_1r_P(\lambda_1)+c_2r_P(\lambda_2)$ 
satisfies the unique
property that 
$r_{P}(c_1\lambda_1+c_2\lambda_2)$ must have in 
Proposition \ref{5.5} (2). 
\end{proof}

\subsection{{\it The group case}} \label{5.7}
Here we note that 
the mapping $r_P$ in the situation of \ref{1.5} 
is the well-known one constructed by Casselman. 
Let $G$ 
be the group $G_1\times G_1$ 
with the involution $\sigma$ which permutes factors. 
Then $\sigma$-split parabolic subgroups are those 
of the form $P_1\times P_1^-$ 
where $P_1$ and $P_1^-$ are 
mutually opposite parabolic subgroups of $G_1$. 
Set $M_1=P_1\cap P_1^-$. 
For an irreducible $\Delta G_1$-distinguished 
representation $\pi_1\otimes\widetilde{\pi_1}$ of 
$G_1\times G_1$, let $\lambda$ be the 
$\Delta G_1$-invariant linear form on 
$\pi_1\otimes\widetilde{\pi_1}$ defined by the canonical 
pairing of  $\pi_1$ and $\widetilde{\pi_1}$ as in \ref{1.5}. 
Then $r_{P_1\times P_1^-}(\lambda)$ is 
an $(M_1\times M_1)\cap\Delta G_1=\Delta M_1$-invariant 
linear form on the Jacquet module 
$(\pi_1\otimes\widetilde{\pi_1})_{P_1\times P_1^-}\simeq 
(\pi_1)_{P_1}\otimes(\widetilde{\pi_1})_{P_1^-}$. 
It is exactly 
the one given by the pairing of the 
Jacquet modules 
$(\pi_1)_{P_1}$ and $(\widetilde{\pi_1})_{P_1^-}$ 
constructed in \cite[4.2.2]{C}. 

\subsection{} \label{5.8}
We study the transitivity 
of the mappings $r_P$ with respect to the 
inclusion of $\sigma$-split parabolic subgroups. 

Let $P=M\ltimes U$ be a $\sigma$-split parabolic 
subgroup of $G$. It is obvious that $\sigma$-split 
parabolic subgroups of $M$ are of the form 
$M\cap Q$ where $Q$ is a $\sigma$-split parabolic 
subgroup of $G$ contained in $P$. Let 
$L$ be the $\sigma$-stable Levi 
subgroup of $Q$. 
It is also the $\sigma$-stable Levi subgroup of 
$M\cap Q$. 
As is well-known, 
$\left(V_P\right)_{M\cap Q}$ is isomorphic to 
$V_Q$ as an $L$-module. Fix an isomorphism and 
identify $\left(V_P\right)_{M\cap Q}$ with $V_Q$ 
from now on. There are induced mappings 
$$
r_P:\left(V^*\right)^H\to 
\left(V_P^*\right)^{M\cap H},\quad 
r_{M\cap Q}:\left(V_P^*\right)^{M\cap H}\to 
\left((V_P)_{M\cap Q}^*\right)^{L\cap H}
$$
and 
$$
r_Q:\left(V^*\right)^H\to 
\left(V_Q^*\right)^{L\cap H}
= \left((V_P)_{M\cap Q}^*\right)^{L\cap H}
$$
of invariant linear forms. 

\begin{Proposition} \label{5.9}
For $P$ and $Q$ as above, one has 
$$
r_{M\cap Q}\circ r_P=r_Q. 
$$
In other words, the diagram 
$$\CD
\left(V^*\right)^H @>r_P>> \left(V_P^*\right)^{M\cap H}\\ 
@V r_Q VV @VV r_{M\cap Q} V \\ 
\left(V_Q^*\right)^{L\cap H} @= 
\left((V_P)_{M\cap Q}^*\right)^{L\cap H}
\endCD
$$
is commutative.
\end{Proposition}  
\begin{proof}
Let $\lambda$ be an element of $(V^*)^H$. We put 
$\overline{\lambda}=r_{M\cap Q}\left(r_P(\lambda)\right)$. 
It is regarded as an $L\cap H$-invariant linear 
form on $V_Q$.  Let $S_L$ be the 
$(\sigma,F)$-split 
component of $Q$. By Proposition \ref{5.5} (1), 
we can choose a positive 
real number 
$\varepsilon\leqq 1$ for each $v\in V$ 
such that both 
$$
\langle\lambda,\pi(s)v\rangle
=
\delta_P^{1/2}(s)
\langle r_P(\lambda),\pi_P(s)j_P(v)\rangle
$$
and 
\begin{eqnarray*}
& &\langle r_P(\lambda),\pi_P(s)j_P(v)\rangle \\
& &=
\delta_{M\cap Q}^{1/2}(s)
\langle r_{M\cap Q}\left(r_P(\lambda)\right),
\pi_{M\cap Q}(s)j_{M\cap Q}\left(j_P(v)
\right)\rangle
\end{eqnarray*}
hold for any 
$s\in S_L^-(\varepsilon)$. 
Identifying $j_{M\cap Q}\left(j_P(v)\right)$ with $j_Q(v)$, we have
\begin{eqnarray*}
\langle\lambda,\pi(s)v\rangle
&=&
\delta_P^{1/2}(s)\delta_{M\cap Q}^{1/2}(s)
\langle \overline{\lambda},\pi_{Q}(s)j_Q(v)\rangle \\
&=&
\delta_{Q}^{1/2}(s)
\langle \overline{\lambda},\pi_Q(s)j_Q(v)\rangle 
\end{eqnarray*}
for all 
$s\in S_L^-(\varepsilon)$. 
This means that $\overline{\lambda}$ satisfies the 
unique property that $r_Q(\lambda)$ must have in 
(2) of Proposition \ref{5.5}. 
\end{proof}

\section{Characterization of relative cuspidality}\label{S6}

We shall give a characterization of 
relative cuspidality in terms of Jacquet modules 
along $\sigma$-split parabolic subgroups. 

\subsection{} \label{6.1} 
Let $(\pi,V)$ be an $H$-distinguished 
admissible 
representation of $G$. For a given non-zero 
$H$-invariant linear form $\lambda\in(V^*)^H$ on 
$V$, we have defined the $(H,\lambda)$-matrix 
coefficient $\varphi_{\lambda,v}$ ($v\in V$) by 
$$
\varphi_{\lambda,v}(g)=\langle\lambda,\pi(g^{-1})v
\rangle. 
$$
Recall that $(\pi,V)$ is said to be 
$(H,\lambda)$-relatively cuspidal if the 
support of 
$\varphi_{\lambda,v}$ is compact modulo $ZH$ for 
all $v\in V$. 

Let $\mathbf{Z}_{\sigma}^-$ denote 
the $(\sigma,F)$-split 
component of $\mathbf{Z}$. We have 
an almost direct product decomposition 
$\mathbf{Z}=\mathbf{Z}_{\sigma}^-\cdot 
(\mathbf{Z}\cap\mathbf{H})$. 
This implies that $Z/Z_{\sigma}^-(Z\cap H)$ is 
finite. Hence $(\pi,V)$ is 
$(H,\lambda)$-relatively cuspidal if and only if 
the support of $\varphi_{\lambda,v}$ 
is compact modulo $Z_{\sigma}^-H$ for 
all $v\in V$. 

A large part of this section is 
devoted to the proof of the following theorem. 

\begin{Theorem} \label{6.2} 
Let $(\pi,V)$ be an $H$-distinguished 
admissible representation 
of $G$ and $\lambda$ a non-zero 
$H$-invariant linear form on $V$. 
Then, $(\pi,V)$ is $(H,\lambda)$-relatively 
cuspidal if and only if $r_P(\lambda)=0$ for any proper 
$\sigma$-split parabolic subgroup $P$ 
of $G$. 
\end{Theorem} 

\indent

First we shall prove the {\it only if} part. 
The proof of the {\it if} part will 
be given in \ref{6.8} after several preparatory lemmas. 

\subsection{\it Proof of the only if part} \label{6.3}
Let $P=M\ltimes U$ 
be a proper $\sigma$-split 
parabolic subgroup of $G$ with 
the $(\sigma,F)$-split component $S$. 
We regard $P$ as a standard one 
$P_I$ with respect to a 
suitable choice of 
$(\mathbf{S}_0,\mathbf{A}_{\emptyset},\Delta)$ 
and a $\sigma$-split subset $I\subset\Delta$ as in 
\ref{2.3}. 
Assume that $(\pi,V)$ is $(H,\lambda)$-relatively 
cuspidal. 
Let $K=K_n$ be any member of the 
family $\{K_n\}$ of \ref{4.3} adapted to 
$(\mathbf{S}_0,\mathbf{A}_{\emptyset},\Delta)$. 
Since $V^K$ is finite dimensional, 
we can take a compact subset $C$ of $G$ so that 
the support of $\varphi_{\lambda,v}$ is contained in 
$Z_{\sigma}^-CH$ 
for all $v\in V^K$. Here let us observe that 
$Z_{\sigma}^-CH\cap S^+(1)$ is contained in some 
subset of $S^+(1)$ which is compact modulo 
$Z_{\sigma}^-$: For a given $s\in Z_{\sigma}^-CH\cap S^+(1)$, 
write $s=zch$ with $z\in Z_{\sigma}^-$, $c\in C$ and $h\in H$. 
Since $h=c^{-1}z^{-1}s$ is fixed by $\sigma$, we have 
$$
\sigma(c)^{-1}zs^{-1}=c^{-1}z^{-1}s, 
$$
thus 
$$
(z^{-1}s)^2=c\sigma(c)^{-1}. 
$$
This shows that $(z^{-1}s)^2$, hence also $z^{-1}s$, 
stays in a compact subset of $S^+(1)$. 
Now we may choose a positive real number 
$\varepsilon <1$ such that 
$$
Z_{\sigma}^-CH\cap S^+(1)\subset 
\left\{ s\in S\bigm| 1\leqq |s^{\alpha}|_F
<\varepsilon^{-1}\,
( \alpha\in\Delta\setminus I)\right\}.
$$
Since $s\in S^-(\varepsilon)$ implies that 
$s^{-1}\notin Z_{\sigma}^-CH\cap S^+(1)$, 
we have 
$$
\langle\lambda,\pi(s)v\rangle
=\varphi_{\lambda,v}(s^{-1})=0
$$
for all $s\in S^-(\varepsilon)$ and $v\in V^K$. 
On the other hand, by (1) of 
Proposition \ref{5.5}, we may choose 
$\varepsilon'$ such that the relation 
$$
\langle\lambda,\pi(s)v\rangle
=\delta_P(s)^{1/2}\langle r_P(\lambda),
\pi_P(s)j_P(v)\rangle
$$
holds for any $s\in S^-(\varepsilon')$ and 
$v\in V^K$. Putting 
$\varepsilon''=\min(
\varepsilon,\varepsilon')$, we have 
$$
\langle r_P(\lambda),
\pi_P(s)j_P(v)\rangle=0
$$
for all $s\in S^-(\varepsilon'')$ and 
$v\in V^K$. This shows that 
$r_P(\lambda)$ vanishes on 
$j_P\left(V^K\right)=(V_P)^{M_K}$. 
Letting $K$ vary in the adapted family 
$\{K_n\}$, we conclude that $r_P(\lambda)=0$ on 
$V_P$. \hfill$\qed$

\subsection{} \label{6.4}
Now we turn to the converse direction. 
Let $\lambda$ be a non-zero $H$-invariant 
linear form on an 
$H$-distinguished admissible 
representation $(\pi,V)$ of $G$. 
Assuming that $r_P(\lambda)=0$ for 
proper $\sigma$-split parabolic subgroups 
$P$ of $G$, we shall investigate 
the support 
of $(H,\lambda)$-matrix 
coefficients of $\pi$. 
First we observe a relation between 
$r_P$ and $r_Q$ when $P$ and 
$Q$ are conjugate. 

\subsection{} \label{6.5}
Fix a data $(\mathbf{S}_0, 
\mathbf{A}_{\emptyset}, \Delta)$ 
as in \ref{2.1} and let 
$\mathbf{P}_0=\mathbf{M}_0\ltimes
\mathbf{U}_0$ be the corresponding 
minimal $\sigma$-split parabolic $F$-subgroup as 
in \ref{2.4}. 
For an element 
$\gamma\in(\mathbf{M}_0\mathbf{H})(F)$, 
we put $m_{\gamma}=
\gamma\sigma(\gamma)^{-1}\in M_0$. Define the 
$F$-involution $\sigma_{\gamma}$ on 
$\mathbf{G}$ by 
$$
\sigma_{\gamma}(g)=m_{\gamma}\sigma(g)m_{\gamma}^{-1}
=\gamma\sigma(\gamma^{-1}g\gamma)\gamma^{-1}. 
$$
The 
$\sigma_{\gamma}$-fixed point subgroup 
in $G$ coincides with 
$\gamma H\gamma^{-1}$, which is denoted by
$H_{\gamma}$. Note that $\mathbf{S}_0$ is also 
a maximal 
$(\sigma_{\gamma}, F)$-split torus. Moreover, any 
$\sigma$-split parabolic subgroup $\mathbf{P}$ 
standard with respect to 
$(\mathbf{S}_0, 
\mathbf{A}_{\emptyset}, \Delta)$ is also 
$\sigma_{\gamma}$-split. The $(\sigma,F)$-split
component and the $(\sigma_{\gamma},F)$-split component 
coincide for such a parabolic subgroup $\mathbf{P}$. 
Now we can define the mapping 
\begin{equation}
r_P:(V^*)^{H_{\gamma}}\to 
(V_P^*)^{M\cap H_{\gamma}} \label{r_P}
\end{equation}
for an $H_{\gamma}$-distinguished representation 
$(\pi,V)$ 
by changing $\sigma$ to $\sigma_{\gamma}$ and 
$H$ to $H_{\gamma}$. 

On the other hand, for $\gamma$ and 
$\mathbf{P}$ as above, put 
$\mathbf{Q}=
\gamma^{-1}\mathbf{P}\gamma$. It is 
$\sigma$-split, but possibly non-standard. 
If $\mathbf{S}$ denotes the $(\sigma,F)$-split component of 
$\mathbf{P}$, then the conjugate 
$\gamma^{-1}\mathbf{S}\gamma$ gives that of 
$\mathbf Q$. We can define the mapping 
\begin{equation}
r_{Q}:(V^*)^H\to 
(V_{Q}^*)^{\gamma^{-1}M\gamma\cap H} \label{r_Q}
\end{equation}
for an $H$-distinguished representation $(\pi,V)$. 
Two mappings \eqref{r_P} and \eqref{r_Q} are related by 
the isomorphism 
$\pi^*(\gamma): (V^*)^H\to 
(V^*)^{H_{\gamma}}$ as follows. 

\begin{Lemma} \label{6.6}
For $\gamma$, $P$ and 
$Q$ as above, the relation 
$$
\langle r_P
\left(\pi^*(\gamma)\lambda\right),j_P(v)\rangle
=\langle r_{Q}
(\lambda),j_{Q}
(\pi(\gamma^{-1})v)\rangle
$$
holds for every $\lambda\in\left(V^*\right)^H$ and 
$v\in V$. 
\end{Lemma} 
\begin{proof}
The mapping $\pi(\gamma^{-1}):V\to V$ 
sends $V(U)$ isomorphically onto 
$V(\gamma^{-1}U\gamma)$, hence induces 
an isomorphism 
$$
\overline{\pi(\gamma^{-1})}: V_P=V/V(U) 
\to V/V(\gamma^{-1}U\gamma)=V_{Q} 
$$
with the relation 
$$
\overline{\pi(\gamma^{-1})}\left(j_P(v)\right)
=j_{Q}\left(
\pi(\gamma^{-1})v\right). 
$$
Under this notation, the right hand side of the lemma 
is written as 
$\langle r_{Q}
(\lambda),\overline{\pi(\gamma^{-1})}
\left(j_P(v)\right)\rangle$. Set 
$\overline{\lambda}=r_{Q}
(\lambda)\circ\overline{\pi(\gamma^{-1})}$. 
We show that $\overline{\lambda}$ satisfies the 
unique property in (2) of Proposition \ref{5.5} 
that $r_P
\left(\pi^*(\gamma)\lambda\right)$ must have. 
For $s$ in the $(\sigma,F)$-split component 
$S$ of $P$, we have 
\begin{eqnarray*} 
& &\delta_P^{1/2}(s)\langle
\overline{\lambda},\pi_P(s)j_P(v)\rangle
=
\langle r_{Q}(\lambda), 
\overline{\pi(\gamma^{-1})}\left(
\delta_P^{1/2}(s)\pi_P(s)j_P(v)\right)\rangle \\ 
& &=\langle r_{Q}(\lambda),
\overline{\pi(\gamma^{-1})}\left(
j_P(\pi(s)v)\right)
=\langle r_{Q}(\lambda),
j_{Q}\left(
\pi(\gamma^{-1})\pi(s)v)\right)\rangle \\
& &=\langle r_{Q}(\lambda),
\delta_{Q}^{1/2}(\gamma^{-1}s\gamma)
\pi_{Q}(\gamma^{-1}s\gamma)
j_{Q}\left(
\pi(\gamma^{-1})v)\right)\rangle. 
\end{eqnarray*}
By Proposition \ref{5.5} (1) applied to $Q$, there exists 
a positive real 
number $\varepsilon\leqq 1$ such that 
the last quantity is equal to 
$$
\langle\lambda,\pi(\gamma^{-1}s\gamma)\pi(\gamma^{-1})v
\rangle
=\langle\pi^*(\gamma)\lambda,\pi(s)v\rangle
$$
for all $s\in S^-(\varepsilon)$. This shows the 
claim. 
\end{proof}

Next we shall see what happens if 
$r_P(\lambda)$ vanishes for a single $P$. 

\begin{Lemma} \label{6.7}
Fix a data $(\mathbf{S}_0, 
\mathbf{A}_{\emptyset}, \Delta)$ and let 
$I\subset\Delta$ be a $\sigma$-split subset, 
$\gamma$ an element of 
$(\mathbf{M}_0\mathbf{H})(F)$ and $\Omega$ 
a compact subset of $G$. 
Let $\lambda$ be a 
non-zero $H$-invariant linear form on 
an admissible representation $(\pi,V)$ of $G$. 
Suppose that 
$r_{\gamma^{-1}P_I\gamma}(\lambda)=0$. 
Then for each $v\in V$, there exists a
positive real number 
$\varepsilon=\varepsilon_{I, \gamma}\leqq 1$ such that 
$\varphi_{\lambda,v}$ 
vanishes 
identically on $\Omega S_I^+(\varepsilon)\gamma H$. 
\end{Lemma} 
\begin{proof} 
We abbreviate $P_I$ and $S_I^-(\varepsilon)$ respectively 
as 
$P$ and $S^-(\varepsilon)$. 
For $k\in\Omega$, $s\in S^-(\varepsilon)$ and 
$h\in H$, we have 
$$
\varphi_{\lambda,v}(ks^{-1}\gamma h)
=\langle\lambda,\pi(\gamma^{-1})\pi(s)\pi(k^{-1})v\rangle
=\langle\pi^*(\gamma)\lambda,\pi(s)\pi(k^{-1})v\rangle. 
$$
Note that $\pi(k^{-1})v$ stays in a finite dimensional subspace 
for any $k\in\Omega$. 
By applying Proposition \ref{5.5} (1) to 
the linear form $\pi^*(\gamma)(\lambda)\in (V^*)^{H_{\gamma}}$, we 
can take a positive real number 
$\varepsilon\leqq 1$ such that
\begin{eqnarray*} 
\langle\pi^*(\gamma)\lambda,\pi(s)\pi(k^{-1})v\rangle
&=&
\delta_{P}^{1/2}(s)
\langle
r_P(\pi^*(\gamma)\lambda),
\pi_{P}(s)j_{P}(\pi(k^{-1})v)\rangle \\
&=& \langle
r_{P}(\pi^*(\gamma)\lambda),
j_{P}(\pi(s)\pi(k^{-1})v)\rangle 
\end{eqnarray*}
for any $s\in S^-(\varepsilon)$ and $k\in\Omega$. 
By Lemma \ref{6.6}, the right hand side is equal to 
$$
\langle r_{\gamma^{-1}P\gamma}
(\lambda),j_{\gamma^{-1}P\gamma}
\left(\pi(\gamma^{-1})\pi(s)\pi(k^{-1})v
\right)\rangle,
$$
which is zero by assumption. 
\end{proof} 

Now we give the rest of the proof of \ref{6.2}.  

\subsection{\it Proof of the if part} \label{6.8}
Recall from Corollary \ref{3.4} that 
$G$ is decomposed as 
$$
G=\Omega S_0^+\Gamma H
$$
for a suitable compact subset 
$\Omega$ of $G$ and a finite 
subset $\Gamma$ of 
$(\mathbf{M}_0\mathbf{H})(F)$. 
Assume that $r_P(\lambda)=0$ for all proper 
$\sigma$-split parabolic subgroup $P$. 
For a given $v\in V$, let $\varepsilon$ be 
the minimum of 
$\varepsilon_{I,\gamma}$ in Lemma \ref{6.7} where 
$I$ runs over all proper 
$\sigma$-split subsets of $\Delta$ and 
$\gamma$ runs over $\Gamma$. Then, 
$\varphi_{\lambda,v}|_{\Omega s^{-1}\Gamma H}
\equiv 0$ if 
$s\in S_0^-(1)\cap S_I^-(\varepsilon)$ for some 
$\sigma$-split subset $I\subset\Delta$. 
In particular, 
$s\in S_0^-(1)$ cannot be in $S_{I_{\alpha}}^-(\varepsilon)$ 
for all $\alpha\in\Delta\setminus\Delta_{\sigma}$ 
if $\varphi_{\lambda,v}$ is not 
identically zero on $\Omega s^{-1}\Gamma H$. 
Here $I_{\alpha}$ is the maximal $\sigma$-split subset of $\Delta$ 
as in \ref{2.6}. Note that 
$s^{\alpha}=s^{\sigma^*(\alpha)}$ for all $s\in S_0$ if 
$\alpha\in\Delta\setminus\Delta_{\sigma}$. As a result, 
the support of 
$\varphi_{\lambda,v}$ is 
contained in the union of 
$\Omega s^{-1}\Gamma H$ 
where $s\in S_0^-(1)$ satisfies 
$\varepsilon<|s^{\alpha}|_F\leqq 1$ for all 
$\alpha\in\Delta\setminus\Delta_{\sigma}$. 
However, the set 
$\{ s\in S_0\bigm| \varepsilon<|s^{\alpha}|_F\leqq 1
\,(\alpha\in\Delta\setminus\Delta_{\sigma})\}$ is finite
modulo 
$Z_{\sigma}^-\cdot\mathbf{S}_0(\mathcal O_F)$. Hence the 
support of 
$\varphi_{\lambda,v}$ is compact 
modulo $Z_{\sigma}^-H$. 

This completes the proof of Theorem \ref{6.2}. 
\hfill$\qed$

\indent

Recall that $(\pi,V)$ is said to be $H$-relatively cuspidal 
if it is $(H,\lambda)$-relatively cuspidal for all 
$\lambda\in (V^*)^H$. Now we have obtained one of our 
main theorem. 

\begin{Theorem} \label{6.9}
An $H$-distinguished 
admissible representation $(\pi,V)$ 
of $G$ is
$H$-relatively cuspidal if and only if 
$r_P((V^*)^H)=0$ for any proper 
$\sigma$-split parabolic subgroup $P$ of $G$. 
\end{Theorem} 

\begin{Remark}\label{6.10} 
If all $\sigma$-split parabolic 
$F$-subgroups are $H$-conjugate to one in a fixed 
standard class (see Lemma \ref{2.5} (2) for example), then by 
Lemma \ref{6.6}, 
it turns out that 
$r_P((V^*)^H)$ vanishes for all proper $\sigma$-split $P$ if 
and only if it does 
for all standard proper $P$. Thus, in such a case, 
the relative cuspidality is 
characterized by the vanishing of $r_P$ only for 
all standard proper $\sigma$-split $P$. 
\end{Remark} 

\subsection{\it The group case} \label{6.11}
The above characterization of the relative cuspidality 
gives the well-known theorem due to Jacquet in 
the group case. Let 
$\pi_1\otimes\widetilde{\pi_1}$ be an irreducible 
$\Delta G_1$-distinguished representation of $G_1\times G_1$ 
with the canonical invariant linear form $\lambda$ (see 
\ref{1.5}). For a $\sigma$-split parabolic $P_1\times P_1^-$, 
the linear form $r_{P_1\times P_1^-}(\lambda)$ is regarded as the 
Casselman's pairing 
of $(\pi_1)_{P_1}$ and $(\widetilde{\pi_1})_{P_1^-}$ 
(see \ref{5.7}). 
It vanishes if and only if $(\pi_1)_{P_1}$ vanishes since 
Casselman's pairing is 
non-degenerate \cite[4.2.4]{C}. Thus, 
Theorem \ref{6.9} in the group case 
asserts that $\pi_1$ is cuspidal 
if and only if $(\pi_1)_{P_1}=0$ for all proper parabolic 
subgroup $P_1$ of $G_1$. 

\section{Relative subrepresentation theorem} \label{S7} 
Here we give a proof of the relative version of 
Jacquet's subrepresentation theorem. From an 
inductive argument using 
the Frobenius reciprocity and the transitivity of 
the mapping $r_P$, it is a natural 
consequence of the characterization theorem in 
the previous section. 

\begin{Theorem} \label{7.1}
Let $\pi$ be an irreducible $H$-distinguished 
admissible 
representation of $G$. Then there exists a 
$\sigma$-split parabolic subgroup $P=M\ltimes U$ 
of $G$ and an irreducible $M\cap H$-relatively
cuspidal representation $\rho$ of $M$ such 
that $\pi$ is a subrepresentation of 
${\rm{Ind}}_P^G(\rho)$. 
\end{Theorem} 
\begin{proof} 
We show this by induction on 
the dimension $r$ of the maximal 
$(\sigma,F)$-split tori of
$G/Z$. If $r=0$, 
then $G/ZH$ is compact by \cite[4.3]{HW}. 
Hence every $H$-distinguished representation is $H$-relatively 
cuspidal. Assume $r>0$. If $(\pi, V)$ is $H$-relatively 
cuspidal, there is nothing to prove. If not, then 
there exists a non-zero $H$-invariant linear form 
$\lambda\in
(V^*)^H$ such that 
$\pi$ is not $(H, \lambda)$-relatively cuspidal. 
It follows from Theorem \ref{6.2} that there exists a 
proper $\sigma$-split 
parabolic subgroup $P$ of $G$ 
such that $r_P(\lambda)\ne 0$. Let $Q=M_Q\ltimes U_Q$ be 
minimal among such. Then by Proposition \ref{5.9} and 
Theorem \ref{6.2}, it is seen that 
the Jacquet module $\pi_Q$ 
is $(M_Q\cap H, r_Q(\lambda))$-relatively 
cuspidal. Apply Proposition \ref{1.11} and take an irreducible 
$M_Q\cap H$-distinguished 
quotient $\rho'$ of $\pi_Q$. 
By the Frobenius reciprocity asserting that 
$$
{\rm{Hom}}_G(\pi,{\rm{Ind}}_Q^G(\rho'))
\simeq
{\rm{Hom}}_{M_Q}(\pi_Q,\rho')\,\left(\ne 0\right),
$$
we have an embedding of $\pi$ into 
${\rm{Ind}}_Q^G(\rho')$. Now the dimension of 
maximal 
$(\sigma,F)$-split tori in 
$M_Q/A_Q$ 
(where 
$A_Q$ denotes 
the $F$-split component of
$M_Q$) is strictly less than $r$. 
By the induction hypothesis 
applied to $\rho'$, 
there exists a $\sigma$-split parabolic 
subgroup 
$M_Q\cap P=M_P\ltimes U_P$ of $M_Q$ and an 
irreducible 
$M_P\cap H$-relatively cuspidal representation $\rho$ 
of $M_P$ such that $\rho'$ is a subrepresentation of 
${\rm{Ind}}_{M_Q\cap P}^{M_Q}(\rho)$. Here 
$P$ is a $\sigma$-split parabolic subgroup of 
$G$ contained in $Q$ and $M_P$ is the 
$\sigma$-stable Levi subgroup of $P$ (see \ref{5.8}). 
As a consequence, 
$\pi$ is a subrepresentation 
of ${\rm{Ind}}_Q^G\left({\rm{Ind}}_{M_Q\cap P}^{M_Q}
(\rho)\right)
\simeq {\rm{Ind}}_P^G(\rho)$. 
\end{proof}

\subsection{{\it The group case.}} \label{7.2}
We shall apply the above theorem to the group case. 
For any irreducible admissible $\Delta G_1$-distinguished 
representation $\pi_1\otimes\widetilde{\pi_1}$ of 
$G_1\times G_1$, there exist a $\sigma$-split 
parabolic subgroup $P_1\times P_1^-$ and an irreducible 
$\Delta M_1$-relatively cuspidal representation 
$\rho_1\otimes\widetilde{\rho_1}$ of $M_1\times M_1$ 
such that $\pi_1\otimes\widetilde{\pi_1}$ can be embedded in 
$$
{\rm{Ind}}_{P_1\times P_1^-}^{G_1\times G_1}
(\rho_1\otimes\widetilde{\rho_1})
\simeq 
{\rm{Ind}}_{P_1}^{G_1}
(\rho_1)\otimes 
{\rm{Ind}}_{P_1^-}^{G_1}
(\widetilde{\rho_1}). 
$$
Here $\rho_1$ is an irreducible cuspidal representation 
of $M_1$. The embedding 
$\pi_1\hookrightarrow {\rm{Ind}}_{P_1}^{G_1}
(\rho_1)$ on the first factor asserts nothing but 
Jacquet's subrepresentation theorem (see \cite[5.1.2]{C}). 
On the second factor we have the embedding 
$\widetilde{\pi_1}\hookrightarrow 
{\rm{Ind}}_{P_1^-}^{G_1}
(\widetilde{\rho_1})$ at the 
same time (see \cite[3.3.1]{S}). 

\section{Examples of relatively cuspidal representations}\label{S8}

This section is devoted to several examples of 
relatively cuspidal distinguished representations. 

\begin{Proposition} \label{8.1} 
Cuspidal 
$H$-distinguished representations are 
$H$-relatively 
cuspidal. 
\end{Proposition} 
\begin{proof} 
This is immediate from Theorem \ref{6.9}. 
\end{proof} 

Certain examples of cuspidal distinguished 
representations were constructed by 
Hakim and Mao \cite{HM1,HM2} 
for the symmetric pairs 
$(GL_n(E), U_n(E/F))$ and 
$(GL_n(F), O_n(F))$. 

In the rest of this section, 
we shall give two examples of 
non-cuspidal but 
relatively cuspidal distinguished representations. 

\subsection{The symmetric pair 
$(GL_n,GL_{n-1}\times GL_1)$ } \label{8.2} 

For $n\geqq 3$, it is known that 
irreducible 
cuspidal representations of 
$GL_n(F)$ do not 
have any non-zero $GL_{n-1}(F)$-invariant 
linear form (see \cite[Proof of Theorem 2]{P}). 
Therefore irreducible cuspidals 
cannot be 
$GL_{n-1}(F)\times 
GL_1(F)$-distinguished. 
We shall construct a class of 
relatively cuspidal representations
below.  For $n=3$ similar representations 
were treated by D. Prasad in \cite{P}. 
Representations attached to 
finite symmetric pair of this type 
was studied in \cite{vD}. 

\subsubsection{} \label{8.2.1}  
Let $\mathbf{G}$ be the group 
$GL_{n/F}$. Put 
$$
\epsilon=\left(\smallmatrix & &1\\
 &1_{n-2}& \\
1& & \endsmallmatrix\right),\quad 
\epsilon'=\left(\smallmatrix 1_{n-1}& \\
 &-1\endsmallmatrix\right),\quad 
\eta=\left(\smallmatrix 1& &1\\
 &1_{n-2}& \\
1& &-1\endsmallmatrix\right).\quad 
$$
We consider the inner involution 
$\sigma={\rm{Int}}
(\epsilon)$
on $\mathbf{G}$. Let $\mathbf{H}$ be the 
$\sigma$-fixed point subgroup in $\mathbf{G}$ and 
$\mathbf{H}'$ the fixed point subgroup of 
the other involution 
${\rm{Int}}(\epsilon')$. We have 
$$
\mathbf{H}'=
\big\{ \left(\smallmatrix A& 0\\ 
0&u\endsmallmatrix\right)\bigm| A\in GL_{n-1},\,
u\in GL_1\big\}. 
$$
By the relation $\epsilon=\eta\epsilon'\eta^{-1}$,  
it is seen that 
$\mathbf{H}=\eta^{-1}
\mathbf{H}'\eta$. 
As a maximal $(\sigma,F)$-split torus, we shall take 
$$
\mathbf{S}_0
=\big\{ {\rm{diag}}(s,1,\cdots, 1,s^{-1})\bigm| 
s\in GL_1
\big\}.
$$
Let $\mathbf{A}_{\emptyset}$ be the maximal 
$F$-split torus consisting of all the diagonal 
matrices and $\Delta$ the set of simple roots of 
$(\mathbf{G},\mathbf{A}_{\emptyset})$ corresponding 
to the Borel subgroup of upper triangular matrices. 
If $e_i\in X^*(\mathbf{A}_{\emptyset})$ (for $1\leqq i\leqq n$)
denotes the $F$-rational character of $\mathbf{A}_{\emptyset}$ 
given by 
$$
e_i\left(
\mathrm{diag}(a_1,a_2,\cdots, a_n)
\right)=a_i,
$$
then $\Delta$ is described as 
$$
\Delta=\{e_i-e_{i+1}\bigm| 1\leqq i\leqq n-1\}. 
$$
It is easy to see that $\Delta$ is a $\sigma$-basis. We have 
$$
\Delta_{\sigma}=\big\{e_2-e_3,
e_3-e_4, \cdots,
e_{n-2}-e_{n-1}\big\}
$$
and 
$$
\sigma(e_1-e_2)=-(e_2-e_n),
\quad 
\sigma(e_{n-1}-e_n)=-(e_1-e_{n-1}). 
$$
There is only one proper standard $\sigma$-split 
parabolic subgroup $\mathbf{P}_0$, 
the minimal one corresponding to 
$\Delta_{\sigma}$. It is the standard 
parabolic subgroup of type 
$(1,n-2,1)$, i.e., 
$$
\mathbf{P}_0=\mathbf{P}_{1,n-2,1}=
\big\{\left(\smallmatrix a& &*\\ 
 &A&\\ 
0& &b\endsmallmatrix\right)\in\mathbf{G}
\bigm | a,b\in GL_1,\,
A\in GL_{n-2} \big\}. 
$$
The $\sigma$-stable Levi subgroup 
$\mathbf{M}_0$ of 
$\mathbf{P}_0$ is given by 
$$
\mathbf{M}_0=\big\{\left(\smallmatrix a& &0\\ 
 &A&\\ 
0& &b\endsmallmatrix\right)\in\mathbf{G}
\bigm | a,b\in GL_1,\,
A\in GL_{n-2} \big\}.  
$$
Note that 
$\sigma$ acts on $\mathbf{M}_0$ 
as 
$$
\sigma\left(\smallmatrix a& &0\\ 
 &A&\\ 
0& &b\endsmallmatrix\right)=
\left(\smallmatrix b& &0\\ 
 &A&\\ 
0& &a\endsmallmatrix\right),
$$ 
hence $\mathbf{M}_0\cap \mathbf{H}$ consists of all the 
matrices of the
form 
$\left(\smallmatrix a& &0\\ 
 &A&\\ 
0& &a\endsmallmatrix\right)$. 
This implies that the Galois cohomology of 
$\mathbf{M}_0\cap \mathbf{H}$ over $F$ is trivial. 
So we have 
$\left(\mathbf{M}_0\mathbf{H}\right)(F)=
M_0H$ and thus every proper $\sigma$-split 
parabolic subgroup of $\mathbf{G}$ is $H$-conjugate to 
$\mathbf{P}_0$ 
in this case by Lemma \ref{2.5} (2). 

\subsubsection{} \label{8.2.2} 
Let $\mathbf{Q}=\mathbf{P}_{n-2,2}$ be the standard parabolic subgroup 
of $\mathbf{G}$ of type $(n-2,2)$, i.e.,  
$$
\mathbf{Q}=\big\{ \left(\smallmatrix A& *\\ 
0&g\endsmallmatrix\right)\bigm| A\in GL_{n-2},\,
g\in GL_2\big\}. 
$$
This is not $\sigma$-split nor $\sigma$-stable, but 
the conjugate $\eta\mathbf{Q}\eta^{-1}$ is 
$\sigma$-stable. 
Take a representation 
$(\rho,W_{\rho})$ of 
$GL_2(F)$ 
and consider the normalized induction 
$$
\pi={\rm{Ind}}_Q^G(1_{GL_{n-2}(F)}
\otimes\rho). 
$$

\begin{PropS8} \label{8.2.3} 
If $\rho$ is an irreducible 
cuspidal representation 
of $GL_2(F)$ with trivial central 
character, then the induced representation 
$$
\pi={\rm{Ind}}_Q^G(1_{GL_{n-2}(F)}
\otimes\rho) 
$$
is irreducible, $H$-distinguished and $H$-relatively 
cuspidal. 
\end{PropS8} 

\indent

The irreducibility follows from the general 
result of \cite[3.2, 4.2]{Z}. 
Up to \ref{8.2.5}, we
show  that 
$\pi$ carries a non-zero $H'$-invariant linear 
form (hence is $H$-distinguished by the relation 
$H=\eta^{-1}H'\eta$). 
Relative cuspidality will be seen in \ref{8.2.6}. 

\subsubsection{} \label{8.2.4} 
Set $\mathcal O=QH'$. This is a closed $(Q, H')$-double coset of 
$G$ since $Q$ is ${\rm Int}(\epsilon')$-stable 
(see \cite[13.3]{HW}). 
Let ${\rm I}(\rho;\mathcal O)$ be the space of all 
locally constant mappings 
$\phi: \mathcal O\to W_{\rho}$ 
satisfying 
\begin{eqnarray*} 
\phi
\left(\left(\smallmatrix A&*\\ 0&g\endsmallmatrix\right)
x\right) &=&
\delta_P\left(\smallmatrix A&*\\ 0&g\endsmallmatrix\right)
^{1/2}
\rho(g)\phi(x)\\ 
&=&|\det(A)|\cdot |\det(g)|^{-(n-2)/2}\cdot\rho(g)\phi(x) 
\end{eqnarray*}
for all $\left(\smallmatrix A&*\\ 0&g\endsmallmatrix\right)
\in Q$ and $x\in \mathcal O$. This is a smooth $H'$-module 
by the right translation. The restriction 
map ${\rm {Ind}}_Q^G(1_{GL_{n-2}(F)}
\otimes\rho) \stackrel{\rm{res}}{\to} 
{\rm I}(\rho;\mathcal O)$ is a surjective $H'$-morphism. 
Taking $H'$-invariants in the dual mapping, we have an 
injection  
$$
\left( {\rm I}(\rho;\mathcal O)^*\right)^{H'}
\hookrightarrow \left({\rm
{Ind}}_Q^G(1_{GL_{n-2}(F)}
\otimes\rho)^*\right)^{H'}. 
$$ 
So it is enough to construct 
a non-zero element of 
$\left( {\rm I}(\rho;\mathcal O)^*\right)^{H'}$. 

\subsubsection{} \label{8.2.5} 
Set $R=Q\cap H'$ and $\rho_R=\left(
(1_{GL_{n-2}(F)}\otimes\rho)
\delta_Q^{1/2}
\right)|_R$. Explicitly the subgroup $R$ is given by 
$$
\font\b=cmsl9 scaled \magstep2 
\def\bigA{\smash{\hbox{\b A}}}
R=\left\{\left(\smallmatrix 
 & & &*&0\\
 &\bigA& &\vdots&\vdots \\
 & & &*&0\\
0&\ldots&0&a&0\\
0&\ldots&0&0&d
\endsmallmatrix\right) \Biggm| A\in
GL_{n-2}(F),\,a,d
\in F^{\times}\right\}
$$
and the representation $\rho_R$ of 
$R$ on $W_{\rho}$ is given by 
$$
\font\b=cmsl9 scaled \magstep2 
\def\bigA{\smash{\hbox{\b A}}}
\rho_R\left(\smallmatrix 
 & & &*&0\\
 &\bigA& &\vdots&\vdots \\
 & & &*&0\\
0&\ldots&0&a&0\\
0&\ldots&0&0&d
\endsmallmatrix\right)=
\big|\det(A)\big|\cdot 
|a|^{-(n-2)/2}\cdot |d|^{(n-2)/2}\cdot 
\rho\left(\smallmatrix a&0\\ 0&d\endsmallmatrix\right). 
$$
It is seen that $\phi |_{H'}$ belongs to the 
unnormalized induction ${\rm{ind}}_R^{H'}
(\rho_R)$ for all $\phi\in{\rm I}(\rho;\mathcal O)$. 
Furthermore, the mapping 
$\phi\mapsto\phi |_{H'}$ gives an isomorphism 
$$
{\rm I}(\rho;\mathcal O)\simeq 
{\rm{ind}}_R^{H'}
(\rho_R)
$$
of $H'$-modules. 
By a version of Frobenius reciprocity 
\cite[2.4.3]{C}, we have isomorphisms
$$
\left({\rm I}(\rho;\mathcal O)^*\right)^{H'}
\simeq 
{\rm{Hom}}_{H'}
\left({\rm{ind}}_R^{H'}
(\rho_R),\Bbb C\right)\simeq 
{\rm{Hom}}_R\left(\rho_R,\Bbb C_{\delta_{R}}\right)
$$
where $\Bbb C_{\delta_{R}}$ is the one 
dimensional representation of $R$ defined 
by the modulus 
character $\delta_{R}$ of $R$. Note that 
$\delta_{R}$ is given by 
$$
\font\b=cmsl9 scaled \magstep2 
\def\bigA{\smash{\hbox{\b A}}}
\delta_R\left(\smallmatrix 
 & & &*&0\\
 &\bigA& &\vdots&\vdots \\
 & & &*&0\\
0&\ldots&0&a&0\\
0&\ldots&0&0&d
\endsmallmatrix\right)
=\big|\det(A)\big|\cdot |a|^{-(n-2)}.  
$$
So the space ${\rm{Hom}}_R\left(
\rho_R,\Bbb C_{\delta_{R}}\right)$ consists of all the 
linear forms $\nu: W_{\rho}\to\Bbb C$ satisfying 
$$
\langle \nu,\rho\left(\smallmatrix a&0\\ 
0&d\endsmallmatrix\right)w\rangle=
\big|a\big|^{-(n-2)/2}\cdot\big|d\big|^{(n-2)/2}
\langle\nu,w\rangle 
$$
for all $w\in W_{\rho}$ and $a$, $d\in F^{\times}$. 
Under the assumption that $\rho$ is 
cuspidal with trivial
central character, we can construct such a 
non-zero linear form 
$\nu: W_{\rho}\to\Bbb C$ by 
using the {\it Kirillov model} 
$\mathcal K_{\rho}$ of 
$\rho$ (see \cite{JL}). 
The space $\mathcal K_{\rho}$ is $\mathcal C^{\infty}_c(F^{\times})$, 
on which the action of the diagonals is described as 
$$
\left(\rho\left(\smallmatrix a&0\\ 0&1\endsmallmatrix
\right)\xi\right)(x)=\xi(ax)
$$
for $\xi\in\mathcal K_{\rho}$, $x\in F^{\times}$. 
Let $\nu$ be the linear form on $\mathcal K_{\rho}$ defined by 
$$
\langle\nu,\xi\rangle=\int_{F^{\times}}\big|x\big|^{(n-2)/2}
\xi(x)d^{\times}x
$$
where $d^{\times}x$ denotes a Haar measure on 
the multiplicative group $F^{\times}$. 
It behaves under the action of 
$\left(\smallmatrix a&0\\ 
0&d\endsmallmatrix\right)=
\left(\smallmatrix d&0\\ 
0&d\endsmallmatrix\right)
\left(\smallmatrix ad^{-1}&0\\ 
0&d\endsmallmatrix\right)$ as 
\begin{eqnarray*}
& &\langle\nu,\rho
\left(\smallmatrix a&0\\ 
0&d\endsmallmatrix\right)
\xi\rangle=\langle\nu,\rho
\left(\smallmatrix ad^{-1}&0\\ 
0&1\endsmallmatrix\right)
\xi\rangle
=\int_{F^{\times}}
\big|x\big|^{(n-2)/2}
\xi(ad^{-1}x)d^{\times}x \\
& &=\int_{F^{\times}}
\big|xa^{-1}d\big|^{(n-2)/2}
\xi(x)d^{\times}x
=\big|a\big|^{-(n-2)/2}\cdot\big|d\big|^{(n-2)/2}
\langle\nu,\xi\rangle. 
\end{eqnarray*}
This is the desired property. 

\subsubsection{} \label{8.2.6} 
Finally we show that $\pi={\rm {Ind}}_Q^G
(1_{GL_{n-2}(F)}
\otimes\rho)$ is $H$-relatively cuspidal. Assume the contrary. 
Recall that any proper $\sigma$-split parabolic subgroup is 
$H$-conjugate to $P_0$. 
By Theorem \ref{7.1}, there exists an irreducible 
$M_0\cap H$-distinguished (relatively cuspidal) 
representation, say $\chi_1\otimes\theta\otimes\chi_2$, 
of $M_0\simeq F^{\times}\times GL_{n-2}(F)\times F^{\times}$ 
such that $\pi$ can be embedded in 
${\rm{Ind}}_{P_0}^G(\chi_1\otimes\theta\otimes\chi_2)$. 
We must have 
$\theta=1_{GL_{n-2}(F)}$ and 
$\chi_2=\chi_1^{-1}$ by the $M_0\cap H$-distinguishedness. 
Now we have an 
embedding 
$$
{\rm {Ind}}_{P_{n-2,2}}^G
(1_{GL_{n-2}(F)}
\otimes\rho)\hookrightarrow 
{\rm{Ind}}_{P_{1,n-2,1}}^G(\chi_1\otimes 
1_{GL_{n-2}(F)}
\otimes\chi_1^{-1}),
$$
which contradicts to the cuspidality of 
$\rho$. 

\subsection{The symmetric pair 
$(GL_{2n},Sp_n)$} \label{8.3}

This kind of symmetric pair 
was studied by Heumos and Rallis 
in \cite{HR} (cf. \cite{BKS} for the finite field case). 
They have shown that 
$Sp_n(F)$-models and (non-degenerate) Whittaker 
models of $G=GL_{2n}(F)$ are 
disjoint. It turns out that there is no 
irreducible cuspidal $Sp_n(F)$-distinguished 
representation, since cuspidals always have 
Whittaker model. In the following, we shall 
see that irreducible $Sp_n(F)$-distinguished 
Langlands quotient representations of 
$GL_{2n}(F)$ constructed in \cite[11.3]{HR} 
are $Sp_n(F)$-relatively cuspidal 
if the inducing representations of 
$GL_n(F)$ are cuspidal in the 
usual sense. 

\subsubsection{} \label{8.3.1} 
Let $\mathbf{G}$ be the group 
$GL_{2n\,/F}$ and $J_n\in\mathbf{G}$ 
the alternating matrix given by 
$$
J_n=\left(\smallmatrix 0&1& & & & & \\
-1&0& & & & &  \\
 & &0&1& & & \\
 & &-1&0& & & \\
 & & & &\ddots& & \\
 & & & & &0&1\\
 & & & & &-1&0
\endsmallmatrix\right).
$$
Consider the involution $\sigma$ on $\mathbf{G}$ defined 
by 
$$
\sigma(g)=J_n{}^tg^{-1}J_n^{-1}. 
$$
The $\sigma$-fixed point subgroup 
$\mathbf{H}$ of $\mathbf{G}$ is 
the symplectic group with respect 
to $J_n$. 
As a maximal $(\sigma,F)$-split torus we take 
$$
\mathbf{S}_0
=\big\{ {\rm{diag}}(s_1,s_1,
s_2,s_2,\cdots, s_n,,s_n)\bigm|s_i\in
GL_1
\big\}. 
$$
Let $\mathbf{A}_{\emptyset}$ be the maximal 
$F$-split torus consisting of all the diagonal 
matrices and $\Delta$ the set of simple roots of 
$(\mathbf{G},\mathbf{A}_{\emptyset})$ corresponding 
to the Borel subgroup of upper triangular matrices. 
Then $\Delta$ is a $\sigma$-basis: We have 
$$
\sigma(e_i-e_{i+1})
=\begin{cases} 
e_i-e_{i+1}\quad\text{if $i$ is odd,} \\ 
-(e_{i-1}-e_{i+2})\quad\text{if $i$ is even} 
\end{cases} 
$$
where the notation is as in \ref{8.2.1}. 
The minimal $\sigma$-split parabolic 
subgroup $\mathbf{P}_0$ corresponding to $\Delta_{\sigma}$ is 
the standard one of type $(2,2,\cdots,2)$. The $\sigma$-stable 
Levi subgroup $\mathbf{M}_0$ of $\mathbf{P}_0$ 
is isomorphic to the product of 
$n$ copies of $GL_2$. Looking at the action of 
$\sigma$ on $\mathbf{M}_0$, it is seen that 
$\mathbf{M}_0\cap\mathbf{H}$ is isomorphic to 
the product of $n$ copies of $SL_2$. 
Again by the triviality of the Galois cohomology of 
$\mathbf{M}_0\cap\mathbf{H}$ and 
Lemma \ref{2.5} (2), every $\sigma$-split parabolic 
subgroup is 
$H$-conjugate to a standard one. 
Maximal standard 
$\sigma$-split 
parabolics are ones of type $(m, 2n-m)$ where 
$m$ is even and  
any maximal $\sigma$-split parabolics are 
$H$-conjugate to such ones. 

\subsubsection{} \label{8.3.2}
Let $J_n'$ be the alternating matrix given by
$$
J_n'=\begin{pmatrix} 
0&1_n\\ -1_n&0\end{pmatrix} 
$$
and $\mathbf{H}'$ the symplectic 
group with respect to $J_n'$. 
There is an element $\eta\in G$ such that $J_n=\eta J_n'{}^t\eta$, 
which gives the relation
$H=\eta^{-1}H'\eta$. Thus an admissible representation
$(\pi,V)$ of $G$ is $H$-distinguished if and only if it is 
$H'$-distinguished. 

\subsubsection{} \label{8.3.3}
Let $\mathbf{Q}=\mathbf{P}_{n,n}$ be the standard parabolic 
subgroup of type $(n,n)$ given by  
$$
\mathbf{Q}=\left\{ \left(\smallmatrix A& *\\ 
0&D\endsmallmatrix\right)\bigm| A,\,D\in GL_n\,
\right\}. 
$$
The conjugate $\eta\mathbf{Q}\eta^{-1}$ is 
$\sigma$-stable. 
Take an irreducible admissible 
representation $\rho$ of $GL_n(F)$ and form 
the normalized induction 
$$
I_{\rho}={\rm{Ind}}_Q^G\left(
\rho\bigl|\det(\cdot)\bigr|^{1/2}
\otimes\rho\bigl|\det(\cdot)\bigr|^{-1/2}\right).
$$
In \cite[11.3.1.2]{HR}, it is shown that 
$I_{\rho}$ is $H'$-distinguished, reducible of length $2$ 
and the unique irreducible quotient 
$\pi_{\rho}$ of $I_{\rho}$ also is 
$H'$-distinguished (hence $H$-distinguished). 

\begin{PropS8} \label{8.3.4}
If $\rho$ is an irreducible 
cuspidal representation 
of $GL_n(F)$, then 
$\pi_{\rho}$ is $H$-relatively cuspidal. 
\end{PropS8} 
\begin{proof} 
By \ref{6.10} and the remark at the end of \ref{8.3.1}, 
it is enough to see that 
$r_P\left((\pi_{\rho}^*)^H\right)=0$ for all 
maximal standard $\sigma$-split parabolic subgroup $P$. 
We show that 
$(\pi_{\rho})_P$ cannot be $M\cap H$-distinguished 
for any maximal 
$\sigma$-split parabolic $P=P_{m,2n-m}=M\ltimes U$. 
Except for the case that $m=n$ and $n$ is even, 
$P$ is not associated to $Q$, hence $(I_{\rho})_P$ 
(and also $(\pi_{\rho})_P$) vanishes by the cuspidality 
of $\rho$ (see \cite[2.13(a)]{BZ}). 
If $n$ is even and $m=n$ (namely $P=Q$), then 
$(I_{\rho})_P$ (and also $(\pi_{\rho})_P$) is 
cuspidal by \cite[2.13(c)]{BZ}. 
In this case, $M=GL_n(F)\times GL_n(F)$ 
and $M\cap H=Sp_n(F)\times Sp_n(F)$, 
so the cuspidal $M$-module $(\pi_{\rho})_P$ cannot be 
$M\cap H$-distinguished by the remark 
at the beginning of \ref{8.3}. 
\end{proof}

\end{document}